\documentclass[a4paper,12pt]{amsart}
\usepackage[utf8]{inputenc}
\usepackage[english]{babel}
\usepackage[a4paper,twoside,top=2.5cm, bottom=2cm, left=2cm, right=2cm]{geometry} 
\usepackage{graphicx}
\usepackage{hyperref}
\usepackage{amsmath}
\usepackage{amsthm}
\usepackage{amssymb}
\usepackage{esint} 
\usepackage{mathrsfs} 
\usepackage{xcolor}
\usepackage{array}
\usepackage{hhline}
\usepackage{enumitem} 
\usepackage{imakeidx} 
\usepackage{xparse} 
\usepackage{mathtools}

\usepackage{comment}

\definecolor{newblue}{rgb}{0.2, 0.3, 0.85}
\hypersetup{colorlinks=true, linkcolor=newblue, citecolor=newblue, urlcolor  = newblue} 

\usepackage{fancyhdr} 
\usepackage{ifthen} 
\usepackage{forloop} 
\usepackage{xstring} 
\usepackage{blindtext} 
\usepackage{nomencl}
\usepackage{dsfont} 
\usepackage{faktor} 
\usepackage{tikz}
\usetikzlibrary{arrows}

\numberwithin{equation}{section}

\usepackage[english,capitalize]{cleveref}

\usepackage[maxbibnames=99,backend=biber, sorting=nyt, doi, url=false]{biblatex}
\usepackage[colorinlistoftodos,textwidth=2.3cm]{todonotes}

\definecolor{dgreen}{rgb}{0.0, 0.56, 0.0}

\newcommand{\N}{\ensuremath{\mathbb N}}
\newcommand{\R}{\ensuremath{\mathbb R}}

\newcommand{\haus}{{\mathcal H}}

\newcommand{\res}{\mathbin{\vrule height 1.6ex depth 0pt width
0.13ex\vrule height 0.13ex depth 0pt width 1.3ex}} 

\DeclarePairedDelimiter\norm{\lVert}{\rVert}
\DeclarePairedDelimiter\scal{\langle}{\rangle}

\newcommand{\st}{\ensuremath{\ :\ }} 
\newcommand{\eqdef}{\ensuremath{\vcentcolon=}}
\newcommand \eps{\ensuremath{\varepsilon}} 
\renewcommand{\epsilon}{\varepsilon}

\newcommand{\de}{\ensuremath{\,\mathrm d}} 
\renewcommand{\d}{\ensuremath{\mathrm d}} 

\newcommand{\RCD}{\mathsf{RCD}}

\newcommand{\dist}{\ensuremath{\mathrm d}}

\DeclareMathOperator{\tr}{tr}

\newcommand{\sff}{\mbox{\rm II}} 
\newcommand{\ric}{\ensuremath{\mathrm{Ric}}} 
\newcommand{\sect}{\ensuremath{\mathrm{Sect}}} 
\DeclareMathOperator{\vol}{vol}
\DeclareMathOperator{\inj}{inj}

\theoremstyle{plain}
\newtheorem{thm}{Theorem}[section] 
\theoremstyle{plain}
\newtheorem{conj}[thm]{Conjecture}
\theoremstyle{plain}
\newtheorem{openq}[thm]{Open Question}
\theoremstyle{plain}
\newtheorem{prop}[thm]{Proposition}
\theoremstyle{plain}
\newtheorem{lemma}[thm]{Lemma}
\theoremstyle{plain}

\theoremstyle{definition}
\newtheorem{defn}[thm]{Definition} 
\theoremstyle{definition}
\newtheorem{remark}[thm]{Remark}
\theoremstyle{definition}

\theoremstyle{definition}

\title[Uniqueness on average of isoperimetric sets]{Uniqueness on average of large isoperimetric sets in noncompact manifolds with nonnegative Ricci curvature}

\author{Gioacchino Antonelli}
\address{Gioacchino Antonelli
\hfill\break Courant Institute Of Mathematical Sciences (NYU), 251 Mercer Street, 10012, New York, USA}
\email{ga2434@nyu.edu}
\author{Marco Pozzetta}\address{Marco Pozzetta \hfill\break Dipartimento di Matematica e Applicazioni ``Renato Caccioppoli'', Universit\'a degli Studi di Napoli Federico II, via Cintia - Monte Sant'Angelo, 80126 Napoli, Italy}\email{marco.pozzetta@unina.it}
\author{Daniele Semola}\address{Daniele Semola \hfill\break ETH Z\"urich, FIM-Forschungsinstitut f\"ur Mathematik, R\"amistrasse 101, 8092, Z\"urich, Switzerland}\email{daniele.semola.math@gmail.com}


\begin{document}

\begin{abstract}
Let $(M^n,g)$ be a complete Riemannian manifold which is not isometric to $\mathbb{R}^n$, has nonnegative Ricci curvature, Euclidean volume growth, and quadratic Riemann curvature decay. We prove that there exists a set $\mathcal{G}\subset (0,\infty)$ with density $1$ at infinity such that for every $V\in \mathcal{G}$ there is a unique isoperimetric set of volume $V$ in $M$; moreover, its boundary is strictly volume preserving stable. 

The latter result cannot be improved to uniqueness or strict stability for every large volume. Indeed, we construct a complete Riemannian surface satisfying the previous assumptions and with the following additional property: there exist arbitrarily large and diverging intervals $I_n\subset (0,\infty)$ such that isoperimetric sets with volumes $V\in I_n$ exist, but they are neither unique nor do they have strictly volume preserving stable boundaries.
\end{abstract}

\maketitle

\tableofcontents

\section{Introduction}

In this paper we study the isoperimetric structure for large volumes of complete Riemannian manifolds $(M^n,g)$ with nonnegative Ricci curvature, Euclidean volume growth, and quadratic Riemann curvature decay. More precisely, we shall assume that $\rm{Ric}\ge 0$ everywhere on $M^n$ and that for some (and hence for every) base point $p\in M$ it holds
\begin{equation}\label{eq:EVG}
   {\rm AVR}(M):=\lim_{r\to \infty}\frac{\mathrm{vol}(B_r(p))}{\omega_n r^n}>0\, ,
\end{equation}
where $\omega_n$ is the volume of the unit ball in $\R^n$, and 
\begin{equation}\label{eq:QCD}
    |\mathrm{Riem}(x)|\le \frac{C}{\d^2(x,p)}\, \quad \text{for all $x\in M$}\, ,
\end{equation}
for some constant $C>0$. This class of manifolds has been studied thoroughly in the literature and includes several examples of interest.
Under these assumptions, it is known (see for instance \cite{MondinoNardulli,AntonelliBrueFogagnoloPozzetta}) that isoperimetric sets with volume $V$ exist for each $V>0$. 

It is well-known that Riemannian manifolds with nonnegative Ricci curvature and Euclidean volume growth have metric cones as blow-downs, namely any pointed sequence $(M, r_i^{-2} g, o)$ for $o \in M$ and diverging $r_i$ admits a subsequence converging in pointed measured Gromov--Hausdorff sense to an $n$-dimensional metric measure cone $(C,\dist_C, \haus^n)$, see \cite{CheegerColding96}. Assuming in addition \eqref{eq:QCD}, then all the blow-downs of $(M^n,g)$ are metric cones $(C(N), \d r^2+r^2g_N)$ over some compact Riemannian manifold $(N,g_N)$, where $g_N$ is a $C^{1,\alpha}$ Riemannian metric for each $\alpha<1$, and $\mathrm{Ric}_N\ge n-2$ (possibly in a weak sense).  
Although the blow-down of $(M^n,g)$ is not necessarily unique under these assumptions, see \cite{Perelman97} and the more recent \cite{ColdingNabercones}, the manifold $(M^n,g)$ is $C^{1,\alpha}$ asymptotic to the family of its blow-down cones, for each $\alpha<1$.
Moreover, for each of these asymptotic cones the isoperimetric set with volume $V$ is unique, it has strictly volume preserving stable boundary, and it coincides with a ball centered at the vertex for each volume $V>0$, see \cite{MorganRitore02}, unless the blow-down is isometric to $\mathbb{R}^n$. However, if a blow-down is isometric to $\mathbb R^n$, then $(M^n,g)$ is necessarily isometric to $\mathbb{R}^n$, see \cite{Coldingvolume}.

Our main result shows that uniqueness and strict volume preserving stability of the isoperimetric sets hold for \emph{most} sufficiently large volumes also for the complete manifold $(M^n,g)$.

\begin{thm}\label{thm:mainthm}
    Let $(M^n,g)$ be a smooth complete Riemannian manifold with $\mathrm{Ric}\ge 0$, Euclidean volume growth \eqref{eq:EVG}, and quadratic Riemann curvature decay \eqref{eq:QCD}. Assume that $(M^n,g)$ is not isometric to $\mathbb{R}^n$. Then there exists a measurable set $\mathcal{G}\subset (0,\infty)$ such that 
    \begin{equation}
        \lim_{r\to \infty}\frac{\mathcal{L}^1(\mathcal{G}\cap (r,2r))}{r}=1\, ,
    \end{equation}
    and for each $V\in \mathcal{G}$ there exists a unique isoperimetric set $E_V\subset M^n$ with $\mathrm{vol}(E_V)=V$. Moreover, the boundary of $E_V$ is a smooth strictly volume preserving stable constant mean curvature hypersurface.  
\end{thm}

\cref{thm:mainthm} is sharp, in the sense that neither the uniqueness nor the strict stability statement can be obtained for \emph{all} sufficiently large volumes in this generality. Indeed, we construct examples of complete surfaces with nonnegative Gaussian curvature, quadratic volume growth, and quadratic curvature decay such that uniqueness and strict stability fail for isoperimetric sets of arbitrarily large volumes.

\begin{thm}\label{thm:counterexamples}
    There exists a complete smooth Riemannian surface $(M^2,g)$ with nonnegative Gaussian curvature, quadratic volume growth \eqref{eq:EVG}, and quadratic curvature decay \eqref{eq:QCD}, such that the following holds. There are disjoint intervals $L_n\subset (0,\infty)$ with $\inf L_n \to \infty$ and $|L_n|\to \infty$ as $n\to\infty$, such that for every $n\in\mathbb N$, and for every $V\in L_n$, there are isoperimetric sets with volume $V$ in $M$ which are neither strictly volume preserving stable, nor unique.
\end{thm}

The construction of \cref{thm:counterexamples} can be easily generalized to obtain analogous examples with $\mathrm{Ric}\ge 0$ satisfying \eqref{eq:EVG} and \eqref{eq:QCD} in any dimension $n\ge 3$. 
\medskip

The question of existence, uniqueness, and characterization of isoperimetric sets for large volumes, and more in general of constant mean curvature hypersurfaces on complete Riemannian manifolds under various types of curvature constraints and asymptotic conditions at infinity has received a great deal of attention in the last thirty years, see for instance \cite{HuiskenYau,QingTian,BrayPhD,Brendle,BrendleEichmairInventiones,BrendleEichmairJDG,EichmairMetzgerinv,EichmairMetzgerJGD,ChodoshEichmairDuke,ChodoshEichmairCrelle,ChodoshEichmairShiYu21,GuanLiWang}, and the references therein. We mention in particular \cite{ChodoshEichmairVolkmann} where the authors obtain the uniqueness of isoperimetric regions for each sufficiently large volume in complete Riemannian manifolds $(M^n,g)$ that are $C^{2,\alpha}$-asymptotic to a fixed cone $C(N^{n-1})$ whose link $(N^{n-1},g_N)$ satisfies 
\begin{equation}\label{eq:link}
\mathrm{Ric}_N\ge n-2\, , \quad \quad \mathrm{vol}(N)<\mathrm{vol}(\mathbb S^{n-1})\, .
\end{equation}
Note that the conditions in \eqref{eq:link} are satisfied for each cross-section of any blow-down of a manifold $(M^n,g)$ as in the assumptions of \cref{thm:mainthm}. On the other hand, the assumptions of \cref{thm:mainthm} are not sufficient to guarantee neither the uniqueness of the blow-down, nor the $C^{2,\alpha}$ convergence.

The mild curvature decay condition \eqref{eq:QCD} makes the techniques of \cite{ChodoshEichmairVolkmann}, which exploit the implicit function theorem leveraging the higher regularity at infinity,
unsuitable for the present setting. 
Moreover, to the best of our knowledge, the \emph{uniqueness on average} captured in \cref{thm:mainthm} had not been studied before.\\
In the rest of the introduction we outline the proof of \cref{thm:mainthm} and highlight the main novelties of our approach.

\subsection*{Strict stability for many volumes}
To set the notation, we recall that the stability (or Jacobi) operator of a smooth two-sided constant mean curvature hypersurface $\Sigma^{n-1}\subset M^n$ takes the form
\begin{equation}\label{eq:stabop}
    C^{\infty}(\Sigma)\ni f\mapsto -\Delta_{\Sigma}f-(\ric(\nu,\nu)+|\sff|^2)f\in C^{\infty}(\Sigma)\, ,
\end{equation}
and it is associated to the quadratic form
\begin{equation}\label{eq:stabform}
    C^{\infty}(\Sigma)\ni f\mapsto I(f,f):=\int_{\Sigma}\left[|\nabla_{\Sigma}f|^2-(\ric(\nu,\nu)+|\sff|^2)f^2\right]\de\vol_{\Sigma}\, .
\end{equation}
In \eqref{eq:stabop} and \eqref{eq:stabform}, $\nu$ and $\sff$ denote a unit normal to $\Sigma$  and its second fundamental form, respectively. If $\Sigma$ is the boundary of an isoperimetric region and $\Sigma$ is smooth, then it is a volume preserving stable constant mean curvature hypersurface. Namely, the stability operator in \eqref{eq:stabop} is nonnegative definite on the subspace of functions with $0$ average on $\Sigma$.\footnote{In order to avoid confusion we stress that in general this is not an invariant subspace under the action of the Jacobi operator.} Equivalently, the quadratic form in \eqref{eq:stabform} on $C^\infty(\Sigma)$ is nonnegative whenever $\int_{\Sigma} f=0$.
We say that such constant mean curvature hypersurface $\Sigma$ is \emph{strictly (volume preserving) stable} provided there exists $\lambda>0$ such that
\begin{equation}
    I(f,f)\ge \lambda\int_{\Sigma}f^2\de\vol_{\Sigma},\, \quad \text{for all} \,  f\in C^{\infty}(\Sigma)\,\,  \text{such that }\, \int_{\Sigma}f\de\vol_{\Sigma}=0\, .
\end{equation}
\smallskip

Under the assumptions of \cref{thm:mainthm}, for each volume $V>0$ it is natural to consider the rescaled pointed manifold 
\begin{equation}
M_V:=(M^n,r_{V}^{-1}\d_g,p) \, ,
\end{equation}
where we set
    \begin{equation}\label{eq:rVintro}
r_V:=\left({\rm AVR}(M) \, \omega_n\right)^{-\frac{1}{n}}V^{\frac{1}{n}}\, ,
    \end{equation} 
with $\omega_n$ the volume of the unit ball in $\mathbb R^n$.
For any sequence $V_i\to \infty$, up to the extraction of a subsequence that we do not relabel, the pointed manifolds $M_{V_i}$ converge to one of the blow-downs $C(N)$ of $(M^n,g)$ in the pointed Gromov-Hausdorff sense and in the $C^{1,\alpha}$ sense locally away from the singular vertex, for each $\alpha<1$, see for instance \cite{CheegerColding96}. 
Denoting by $E_i$ a sequence of isoperimetric sets in the rescaled manifold $M_{V_i}$ with volume equal to ${\rm AVR}(M)\omega_n$ in the rescaled metric, a by now standard argument shows that the isoperimetric boundaries $\partial E_i$ converge in the Hausdorff sense to the boundary of $B_1(o)$, the ball of radius $1$ centered at the vertex of the limit cone $C(N)$, along this converging sequence. See \cref{lemma:hausconv} below for the precise statement and \cite{AntonelliPasqualettoPozzettaSemola1} for a proof in a more general setting. For comparison, let us mention that in the literature about the isoperimetric problem in the asymptotically flat case, see for instance \cite{EichmairMetzgerinv,ChodoshEichmairShiYu21}, proving that rescalings of isoperimetric sets converge to a centered ball for large volumes is a central difficulty.

Note that on a cone the Ricci curvature in the radial direction vanishes, and balls centered at the vertex are totally umbilical. Hence, it follows from \eqref{eq:link} and the Lichnerowicz and Obata theorems that $\partial B_1(o)$ is strictly stable, unless the cone is isometric to $\R^n$.
Ideally, we would like to prove that the stability operators (under volume preserving variations) of $\partial E_i\subset M_{V_i}$ converge in some sense to the stability operator of the limit isoperimetric region $B_1(o)\subset C(N)$ in the blow-down, or at least that
their first eigenvalues are lower semicontinuous with respect to $i$.
The lower semicontinuity would be clearly enough to prove the strict stability of $\partial E_i$ for any sufficiently large $i\in\mathbb{N}$. However,
\cref{thm:counterexamples} shows that this is too optimistic in the present generality. We shall see that the main issue is related to the Ricci curvature term of the stability operator, which should be viewed as an error term and does not converge to $0$ in general. 
The key for us will be that the Ricci curvature term goes to $0$ on average with respect to the volume, see \cref{lemma:Ricciiso0} for a precise statement. Establishing this convergence requires a new method with respect to the above mentioned list of references. Indeed, for the proof of \cref{lemma:Ricciiso0} we will reverse the usual argument for proving the concavity of the isoperimetric profile via the second variation formula, see \cite{BavardPansu,BrayPhD,Bayle04,AntonelliPasqualettoPozzettaSemola1}.

\subsection*{Uniqueness for many volumes}
The proof of the uniqueness part of \cref{thm:mainthm} is based on the broad principle that strictly stable implies locally uniquely minimizing (possibly under a volume constraint). This principle was originally put forward for minimal hypersurfaces in \cite{Whiteminimax}, and later extended to constant mean curvature hypersurfaces in \cite{GrosseBrauckman} and \cite{MorganRos}. However, the implementation in the present setting requires some new insights. Most importantly, the principle cannot be implemented without restricting the set of volumes, as the examples constructed in \cref{thm:counterexamples} illustrate.
Indeed, the blow-down argument discussed above shows that, for $V$ large enough, if there exist two different isoperimetric regions $E_V,E_V'\subset M$ with volume $V$ then they must be close to each other in a scale-invariant sense and this scale-invariant distance converges to $0$ as $V\to\infty$. If both $E_V$ and $E'_V$ have strictly stable boundaries, then by \cite{GrosseBrauckman,MorganRos} they have neighborhoods $\partial E_V\subset U$ and $\partial E'_V\subset U'$ where they are uniquely isoperimetric. The difficulty stems from the fact that the size of these neighborhoods might not be scale invariantly bounded away from $0$. Hence it is not clear whether $\partial E_V'\subset U$ or $\partial E_V\subset U'$. We shall see that it is possible to effectively estimate the size of these neighborhoods, up to further restricting the set of \emph{good} volumes. This estimate requires a new method with respect to those present in the literature, due to the weak convergence of $(M^n,g)$ to its tangent cones under rescaling.
\medskip

There will be three main steps in the proof of the uniqueness. The first step will be to show that if there are two distinct isoperimetric regions $E_V,E_V'$ with the same sufficiently large volume $V$ and the same mean curvature of the boundary, then $\partial E_V'$ can be written as a normal graph over $\partial E_V$ with effective scale-invariant bounds on the graphing function. See \cref{prop:graphfunction} for the precise statement. This step requires some effective versions of classical estimates in geometric measure theory.

The second step of the proof will amount to getting an effective control on the Ricci curvature of $(M^n,g)$ in the direction perpendicular to the boundary of an isoperimetric set, for each isoperimetric set of volume $V$ and for most large volumes $V$. To accomplish this, we use an argument involving estimates on the maximal function of the second derivative of a power of the isoperimetric profile, see the proof of \cref{lemma:maxRicci}. To the best of the authors' knowledge, the only other instance in the literature where the isoperimetric profile is used to show the uniqueness of isoperimetric sets is \cite{BrayPhD}. However, in that case, there is exact Schwarzschild symmetry outside a compact set, and the isoperimetric profile satisfies an ODE.

In the last step we will effectively perform a Taylor expansion of the perimeter functional near $\partial E_V$ by integrating the Laplacian of the signed distance function $\d^s_{E_V}$ in the region \emph{between} $E_V$ and $E'_V$. We note that \emph{effective} Taylor expansions of the area functional have been employed with a different aim in the recent \cite{KetoverMarquesNeves}.
The approach of this paper seems to be novel, even though the relationship between the Laplacian of the distance function, the mean curvature and the stability operator has been used extensively in the last decades, see for instance \cite{MeeksPerezRossurvey,MeeksPerezRoslamination,RosRosenberg}. Setting 
\begin{equation}
u:=w-\frac{1}{P(E_V)}\int_{\partial E_V}w\, \d P_{E_V}\, ,
\end{equation}
with $w$ denoting the graphing function obtained in the first step described above, the main terms in the Taylor expansion of the perimeter will be the stability operator $I_{\partial E_V}(u,u)$ and an error term involving the Ricci curvature of $M$. While the Taylor expansion holds for every volume, we will be able to exploit the result of the second step to show that for most large volumes the Ricci curvature term is negligible. Hence we can deduce uniqueness of isoperimetric sets for such volumes. Indeed, if for such a volume $E_V$ and $E_V'$ are distinct isoperimetric sets, observe they have the same perimeter. Since the stability operator $I_{\partial E_V}(u,u)$ is the dominating term in the effective Taylor expansion of the perimeter, the strict stability of $\partial E_V$ yields a contradiction.

\subsection*{Conjectures and Open Questions}
We end this introduction by mentioning a few open questions that arose during the present work and might be worthwhile of future investigation.

\medskip

In the case of surfaces, i.e., for $n=2$, a slight variant of the argument discussed in \cref{sec:strictstab} allows to prove that most isoperimetric sets have strictly volume preserving stable boundary without assuming that the Gaussian curvature decays quadratically. In this regard we raise the following:

\begin{conj}
  If $n=2$, \cref{thm:mainthm} holds true without assuming that the Gaussian curvature decays quadratically.
\end{conj}

We do not venture in any conjecture in higher dimensions, however we record the following question. We recall the following terminology. We say that a metric space $X$ \textit{splits a line isometrically} if there is a metric space $Y$ such that $X$ is isometric to the metric product $\mathbb R\times Y$.

\begin{openq}
 Let $(M^3,g)$ be a complete Riemannian manifold with nonnegative Ricci curvature and Euclidean volume growth such that no blow-down splits a line isometrically. Understand whether there exists a measurable set $\mathcal{G}\subset (0,\infty)$ such that 
    \begin{equation}
        \lim_{r\to \infty}\frac{\mathcal{L}^1(\mathcal{G}\cap (r,2r))}{r}=1\, ,
    \end{equation}
    and, for each $V\in \mathcal{G}$, and every isoperimetric set $E_V\subset M^n$ with $\mathrm{vol}(E_V)=V$, the boundary of $E_V$ is a smooth strictly volume preserving stable constant mean curvature hypersurface.
\end{openq}

The examples that we construct in the proof of \cref{thm:counterexamples} have the drawback that, although for many large volumes the isoperimetric set with that volume is not unique, all the isoperimetric sets that we can find are isometric to each other.  

\begin{openq}\label{opq:notisometric}
  Understand whether there exists a complete Riemannian manifold $(M^n,g)$ with nonnegative Ricci curvature, Euclidean volume growth, quadratic curvature decay, and such that there exists a sequence $V_i\to \infty$ with the following property. For every $i\in\N$ there exist two distinct isoperimetric sets $E_i,E'_i\subset M$ with volume $V_i$ such that $\partial E_i$ and $\partial E_i'$ are not isometric.  
\end{openq}

By a well-known argument (see, e.g., \cite[Remarque 2.3.2]{BayleThesis} and \cite{Pansuprofile}), a positive answer to \cref{opq:notisometric} would follow from a positive answer to the following:

\begin{openq}
     Understand whether there exists a complete Riemannian manifold $(M^n,g)$ with nonnegative Ricci curvature, Euclidean volume growth, quadratic curvature decay, and such that there exists a sequence $V_i\to \infty$ with the following property. For every $i\in\N$ the isoperimetric profile of $M$ is not differentiable at the point $V_i$.
\end{openq}

We address the reader to \cite{Pansuprofile} for some related constructions in the compact case.

\subsection*{Acknowledgments}
The authors are very grateful to Elia Bru\`e for several conversations, and the generous help he provided them to complete this work.

The authors are grateful to Otis Chodosh for a stimulating email exchange that partly motivated the present paper, to Anton Petrunin for providing a proof of \cref{prop:reachbound}, and to Guido De Philippis for useful suggestions concerning \cref{appA}, and for having pointed out the reference \cite{AllardProceedings86}.

This project started during the Thematic Program on Nonsmooth Riemannian and Lorentzian Geometry organized at the Fields Institute in Fall 2022. The authors gratefully acknowledge the warm hospitality and the stimulating atmosphere.

G.A. acknowledges the hospitality of the IAS during his visit in December 2021 where he discussed topics related to this work with Elia Bruè. G.A. also acknowledges the financial support of the Courant Institute and the AMS-Simons Travel Grant.

M.P. is a member of INdAM - GNAMPA and he is partially supported by the PRIN Project 2022E9CF89 - PNRR Italia Domani, funded by EU Program NextGenerationEU.

D.S. was supported by the FIM ETH Z\"urich through a Hermann Weyl Instructorship and by the Fields Instutute through a Marsden Fellowship. He acknowledges SISSA for the stimulating atmosphere and the support during a visit in Fall 2023.

\section{The example: Proof of \cref{thm:counterexamples}}

This section is dedicated to the proof of \cref{thm:counterexamples}, that we restate below for the ease of readers.

\begin{thm}\label{thm:counterexamplesmain}
    There exists a complete smooth Riemannian surface $(M^2,g)$ with nonnegative Gaussian curvature, quadratic volume growth \eqref{eq:EVG}, and quadratic curvature decay \eqref{eq:QCD}, such that the following holds. There are disjoint intervals $L_n\subset (0,\infty)$ with $\inf L_n \to \infty$ and $|L_n|\to \infty$ as $n\to\infty$, such that for every $n\in\mathbb N$, and for every $V\in L_n$, there are isoperimetric sets with volume $V$ in $M$ which are neither strictly volume preserving stable, nor unique.
\end{thm}

The examples will be warped products over $\mathbb S^1$. We will show that it is possible to choose the warping function in such a way that balls centered at the origin of the warped product do not have volume preserving stable boundaries and all the other requirements are met. In particular, balls centered at the origin cannot be isoperimetric sets for their enclosed volume. The fact that isoperimetric sets are not unique and not strictly volume preserving stable will follow immediately by the existence of an $\mathbb S^1$-symmetry. We will need the following elementary lemma, whose proof is postponed at the end of the section.

\begin{lemma}\label{lem:Technical}
    Let $(a_n),(b_n)$ be two sequences of positive real numbers such that: 
\begin{itemize}
    \item[i)] $a_n$ is diverging, $a_{n+1}>a_{n}+b_{n}$, $a_{n}\geq b_n$;
    \item[ii)] there holds
    \[
    \sum_{n=0}^{\infty} \frac{b_n}{a_n}\leq \frac{1}{16}\, .
    \]
\end{itemize}
Denote $I_n := [a_n,a_n+b_n]$ and let $J_n\subset\subset I_n$. Then there exists a smooth concave function $\varphi:[0,\infty)\to \mathbb R$, with $\varphi(r)=r$ in a neighborhood of $0$, such that for every $n$ there holds
\begin{equation}\label{eqn:ControlAVRStatement}
    \varphi'(r)\geq \frac{7}{8}, \qquad \varphi(r) \geq \frac{7}{8}r\, .
    \end{equation}
\begin{equation}
|\varphi''(r)|\leq \frac{2}{a_n}\quad \forall r\in I_n\, , \quad |\varphi''(r)|\geq \frac{1}{a_n} \quad \forall r\in J_n\, , \quad \varphi''(r)\equiv 0 \quad \forall r\in (0,\infty)\setminus\bigcup_{n=0}^{\infty} I_n\, ,
\end{equation}
\begin{equation}\label{eqn:StabilityWrong}
\frac{\varphi(r)\varphi''(r)-(\varphi'(r))^2+1}{\varphi^2(r)}<0 \quad \forall r\in J_n\, ,
\end{equation}
\begin{equation}\label{eqn:CurvatureDecay}
-\frac{\varphi''(r)}{\varphi(r)}\leq \frac{16}{r^2} \quad \forall r\in (0,\infty)\, .
\end{equation}
\end{lemma}

\begin{proof}[Proof of \cref{thm:counterexamplesmain}]
 Let $(a_n),(b_n)$ be two sequences of positive real numbers satisfying the assumptions of \cref{lem:Technical}. We can further assume that $(b_n)$ diverges. Let $I_n, J_n, \varphi$ be given by \cref{lem:Technical}. Note that, in particular,  $\inf I_n=a_n\to \infty$, and $|I_n|=b_n\to \infty$. 
 
 We consider the manifold $M$ given by the warped product $[0,\infty)\times \mathbb S^1$ endowed with the metric $g:=\mathrm{d}r^2+\varphi(r)^2g_{\mathbb S^1}$. The sectional curvatures of the rotationally symmetric warped products are computed in \cite[Section 4.2.3]{Petersenbook}. The fact that $\varphi$ is concave and \eqref{eqn:CurvatureDecay} holds imply that $\mathrm{Sect}\geq 0$ and $\mathrm{Sect}=O(r^{-2})$. Applying for example \cite[Corollary 5.9]{AFP21} (see also the earlier \cite{Ritore2d} and the more recent \cite[Theorem 1.4]{AntonelliPozzettaAlexandrov}) we get that isoperimetric sets exist for any volume on $M$. Moreover, since we also have $\varphi(r)\geq 7r/8$ (see \eqref{eqn:ControlAVR}), we get that $M$ has Euclidean volume growth \eqref{eq:EVG}.

    By
    \eqref{eqn:StabilityWrong}, we see that all the balls $B_r(0)$ with $r\in J_n$ have unstable boundaries with respect to volume-preserving variations. Hence, in particular, they are not isoperimetric. Indeed, since $\partial B_r(0)$ is isometric to a circle of length $2\pi \varphi(r)$, there exists $f \in C^\infty(\partial B_r(0))$ with $\int_{\partial B_r(0)} f =0$ and $f\not\equiv 0$ such that
    \[
    \int_{\partial B_r(0)} |\nabla f|^2 = \frac{1}{\varphi^2(r)} \int_{\partial B_r(0)} f^2\, ,
    \]
    where the gradient is understood with respect to the submanifold $\partial B_r(0)$.
    Recalling that $\sff_{\partial B_r(0)}=\frac{\varphi'(r)}{\varphi(r)}(g-\mathrm{d}r^2)$ on the tangent space along $\partial B_r(0)$, we find
    \[
        \int_{\partial B_r(0)} |\nabla f|^2 -\left( \left| \sff_{\partial B_r(0)} \right|^2 + \ric(\nu,\nu)\right)f^2 
    =  \int_{\partial B_r(0)} \left( \frac{1-(\varphi'(r))^2 + \varphi(r) \varphi''(r)}{\varphi^2(r)} \right)   f^2 <0\, ,
    \]
    due to \eqref{eqn:StabilityWrong}.
    
    Let $L_n \eqdef \{ |B_r(0)| \st r \in J_n\}$, where we denote by $|B_r(o)|$ the Riemannian volume measure of $B_r(o)$. Isoperimetric sets with volume $V\in L_n$ always exist as we argued above. However they cannot be balls centered at $o$ because $B_r(0)$ is not volume-preserving stable for $r\in J_n$. Since $M$ is rotationally symmetric, rotating isoperimetric sets with volume $V \in L_n$ shows that isoperimetric sets of such volumes are not (even locally) unique, and that they are not strictly volume preserving stable.
    
    Finally, notice that $\inf L_n\to \infty$. Moreover, since $(b_n)$ diverges, we can choose $J_n=[a'_n, b'_n]$ for suitable $a'_n, b'_n$ such that $b'_n-a'_n\to\infty$. Thus it is readily checked that $|L_n|=2\pi\int_{a_n'}^{b_n'}\varphi(r)\de r\to\infty$ as well, because $\varphi(r)\geq 7r/8$ (see \eqref{eqn:ControlAVR}).
\end{proof}

\begin{proof}[Proof of \cref{lem:Technical}]
    We choose $\varphi(r)=r$ in a right neighborhood of zero. We now explicitly construct $\varphi''$ in such a way $\varphi''(r)\leq 0$ everywhere and such that it is nonzero only on the intervals $I_n$. On every $I_n$ we take $\varphi''$ equal to a negative smooth function such that $|\varphi''|=-\varphi''\leq \frac{2}{a_n}$ on $I_n$ and $|\varphi''|=-\varphi''\geq \frac{1}{a_n}$ on $J_n$. In particular
    \[
    \int_{I_n} \varphi'' \geq -\frac{2b_n}{a_n}\, .
    \]
    In this way, for every $r\in (0,\infty)$, we have
    \[
    \varphi'(r)-\varphi'(0)=\int_0^r \varphi'' \geq -\sum_{n=0}^{\infty} \frac{2b_n}{a_n}\, ,
    \]
    hence, using $\varphi(0)=0$, and $\varphi'(0)=1$,
    \begin{equation}\label{eqn:ControlAVR}
    \varphi'(r)\geq \frac{7}{8}, \qquad \varphi(r) \geq \frac{7}{8}r\, .
    \end{equation}
    Let us now verify \eqref{eqn:CurvatureDecay}. The inequality is trivially satisfied outside $I_n$. For $r\in I_n$, since in particular $a_n\leq r\leq a_n+b_n$, and $a_n\geq b_n$, we have that $a_n\geq r/2$, and then
    \[
    -\frac{\varphi''(r)}{\varphi(r)} \leq \frac{16}{7a_nr} \leq \frac{16}{r^2}.
    \]
    Let us now verify \eqref{eqn:StabilityWrong}. For $r\in J_n$ we have 
    \[-\varphi''(r)\varphi(r)\geq \frac{7r}{8a_n} \geq \frac{7}{8}\, ,
    \]
    and thus 
    \[
    \varphi(r)\varphi''(r)-(\varphi'(r))^2+1 \leq -\frac{7}{8}-\left(\frac{7}{8}\right)^2+1<0\, ,
    \]
    which implies \eqref{eqn:StabilityWrong}.
\end{proof}

\section{Strict stability for many volumes}\label{sec:strictstab}

The goal of this section is to prove the strict stability part of \cref{thm:mainthm}. The statement will be important later when it comes to proving uniqueness for most large volumes.

We state below the precise result for the ease of readers.

\begin{thm}\label{thm:strictstab}
    Let $(M^n,g)$ be a smooth complete Riemannian manifold with $\mathrm{Ric}\ge 0$, Euclidean volume growth \eqref{eq:EVG}, and quadratic Riemann curvature decay \eqref{eq:QCD}. Assume that $(M^n,g)$ is not isometric to $\mathbb{R}^n$. Then there exist $\epsilon_M>0$ and a measurable set $\mathcal{G}\subset (0,\infty)$ such that 
    \begin{equation}
        \lim_{r\to \infty}\frac{\mathcal{L}^1(\mathcal{G}\cap (r,2r))}{r}=1\, ,
    \end{equation}
    and for each $V\in \mathcal{G}$ the first eigenvalue of the Jacobi operator for volume preserving variations on the boundary of every isoperimetric set $E_V\subset M$ of volume $V$ is at least $\eps_M V^{-2/n}$.
    In particular, $\partial E_V$ is a (smooth) strictly volume preserving stable constant mean curvature hypersurface.
\end{thm}

There are three intermediate steps for the proof of \cref{thm:strictstab}, corresponding to the following statements: 

\begin{lemma}\label{lemma:hausconv}
    Let $(M^n,g)$ be a smooth complete Riemannian manifold with $\mathrm{Ric}\ge 0$, Euclidean volume growth \eqref{eq:EVG}, and quadratic Riemann curvature decay \eqref{eq:QCD}, and assume that $M$ is not isometric to $\R^n$. Let $p\in M$ be a fixed base point. For every $\epsilon>0$ there exists $V_0=V_0(\epsilon)>0$ such that for every $V>V_0$ and for every isoperimetric set $E_V\subset M$ of volume $V$, setting 
    \begin{equation}\label{eq:rV}
r_V:=\left(\rm{AVR}(M) \, \omega_n\right)^{-\frac{1}{n}}V^{\frac{1}{n}},
    \end{equation}   
    the following hold:
    \begin{equation}\label{eq:Hausconv}
        \d_{\rm{H}}(B_{r_V}(p),E_V)\le \epsilon r_V\, ,\quad \d_{\rm{H}}(\partial B_{r_V}(p),\partial E_V)\le \epsilon r_V\, ,
    \end{equation}
    where $\d_{\rm{H}}$ denotes Gromov--Hausdorff distance.
\end{lemma}

\begin{prop}\label{prop:secondffstable}
    Let $(M^n,g)$ be a smooth complete Riemannian manifold with $\mathrm{Ric}\ge 0$, Euclidean volume growth \eqref{eq:EVG}, and quadratic Riemann curvature decay \eqref{eq:QCD}. Then, for every $\epsilon>0$ there exists $V_0=V_0(\epsilon)>0$ such that for every $V>V_0$ any isoperimetric set $E_V\subset M$ with volume $V$ is smooth and satisfies
\begin{equation}\label{eqn:ControlSecondForm}
 \sup_{x \in \partial E_V} \left| V^{\frac2n} \, |\sff_{\partial E_V}|^2(x) - (n-1)(\omega_n {\rm AVR}(M))^{\frac2n} \right| \le \epsilon\, ,
\end{equation}
where $\sff_{\partial E_V}$ is the second fundamental form of $\partial E_V$.
\end{prop}

\begin{prop}\label{lemma:Ricciiso0}
 Let $(M^n,g)$ be a complete smooth Riemannian manifold with $\ric\ge 0$, and Euclidean volume growth \eqref{eq:EVG}. For every $\epsilon>0$ there exists a measurable set $\mathcal{G}_{\epsilon}\subset (0,\infty)$ such that 
   \begin{equation}
        \lim_{r\to \infty}\frac{\mathcal{L}^1(\mathcal{G}_{\epsilon}\cap (r,2r))}{r}=1\, ,
    \end{equation}
 and for every volume $V\in \mathcal{G}_{\epsilon}$ and every isoperimetric set $E_V\subset M$ with volume $V$ (if there exists one), it holds
\begin{equation}\label{eq:estRiccistab}
    V^{\frac{2}{n}} \fint_{\partial E_{V}} \ric(\nu_{\partial E_V},\nu_{\partial E_V}) \le \epsilon ,
\end{equation}
where $\nu$ denotes the choice of a unit normal.
\end{prop}

Note that \cref{lemma:Ricciiso0} does not require the quadratic curvature decay \eqref{eq:QCD}, and it makes sense as soon as isoperimetric sets exist for sufficiently large volumes. However, if we assume that the curvature decays quadratically, the $L^1$ estimate for the Ricci curvature in the normal direction can be immediately improved to an $L^p$ estimate for every $1\le p<\infty$, up to a multiplicative constant independent of the volume $V$.
\medskip

We start by proving \cref{lemma:hausconv} below, as this will provide us with some relevant terminology for the proof of \cref{thm:strictstab}.

\begin{proof}[Proof of \cref{lemma:hausconv}]
 We claim that under the present assumptions $(M,g)$ does not split any line, and no blow-down of $(M,g)$ splits any line. Given the claim,
 the statement follows from \cite[Theorem 1.2]{AntonelliPasqualettoPozzettaSemola2}, and \cite[Theorem 4.23]{AntonelliPasqualettoPozzettaSemola1}. The rest of the proof is aimed at establishing the claim.
\medskip

It is sufficient to prove that no blow-down of $(M,g)$ splits any line, as if $(M,g)$ splits a line then every blow-down must also split a line. 

Since $(M,g)$ has Euclidean volume growth, every blow-down is a metric cone $C(Z)$ over an $(n-1)$-dimensional metric space $(Z,\dist_Z)$. The latter comes from \cite{CheegerColding96}, see also the previous \cite{BandoKasueNakajima} for the same conclusion under the quadratic curvature decay assumption.
By the quadratic curvature decay, $(Z,\dist_Z)$ is a $C^{1,\alpha}$ Riemannian manifold for every $\alpha<1$. If $C(Z)$ splits a line then $Z$ is a spherical suspension over an $(n-2)$-dimensional metric space $(Z',\dist_{Z'})$. Since $(Z,\dist_Z)$ is a $C^{1,\alpha}$ Riemannian manifold, all points $z\in Z$ are regular, i.e., the tangent cone at each point is isometric to $\R^{n-1}$. In particular, the suspension points are regular and hence $(Z',\dist_{Z'})$ must be isometric to the standard sphere $\mathbb S^{n-2}$ endowed with a metric with constant sectional curvature $1$. Therefore $Z$ is isometric to the standard sphere $\mathbb S^{n-1}$ endowed with a standard metric with constant sectional curvature $1$, and hence $C(Z)$ is isometric to $\R^n$. By volume convergence and volume rigidity, see \cite{Coldingvolume}, $(M,g)$ is also isometric to $\R^n$, a contradiction.
\end{proof}

Before giving detailed proofs of \cref{prop:secondffstable} and \cref{lemma:Ricciiso0}, we discuss how to complete the proof of \cref{thm:strictstab} by taking them for granted.

\begin{proof}[Proof of \cref{thm:strictstab}]

    The proof will be divided in three main steps. In Step 1 we argue that the stability operators of the unit spheres have a uniform spectral gap among all blow-downs of $(M,g)$. In Step 2 we rely on \cref{lemma:hausconv} and \cref{prop:secondffstable} to show that for any sequence of volumes $V_i\to \infty$ and any sequence of isoperimetric sets $E_{i}$ such that $\vol(E_i)=V_i$ the boundaries $\partial E_i$ with the induced Riemannian metrics from the ambient $M$ converge in the $C^{1,\alpha}$ sense to a cross-section of a blow-down of $M$, up to rescaling. In Step 3 we show that, roughly speaking, the first eigenvalue of the stability operator is lower semicontinuous along any such converging sequence provided that the Ricci curvature term is sufficiently small. The proof will be completed with the help of \cref{lemma:Ricciiso0}, whose aim is exactly to control that Ricci curvature term for many volumes.
    \medskip

    \textbf{Step 1.} We claim that there exists $\epsilon_M>0$ such that for every cross-section $(N,g_N)$ of a blow-down $C(N)$ of $(M,g)$ the first non-zero eigenvalue of the Laplacian $\lambda_1(N)$ satisfies $\lambda_1(N)\ge n-1+2\epsilon_M$. 

    As we argued in the proof of \cref{lemma:hausconv} above, any cross-section of a blow-down of $(M,g)$ is a $C^{1,\alpha}$ Riemannian manifold $(N,g)$ with Ricci curvature bounded from below by $n-2$ in the sense of $\RCD$ spaces or, equivalently given the $C^{1,\alpha}$ regularity, in the sense of distributions. Moreover, no such cross-section is isometric to the standard sphere $\mathbb S^{n-1}$ with constant sectional curvature equal to $1$, otherwise $C(N)$ and hence $(M,g)$ would be isometric to $\R^n$, by \cite{Coldingvolume}. By the Lichnerowicz estimate for the spectral gap of the Laplacian, see for instance \cite[Theorem 4.22]{ErbarKuwadaSturm} for the setting of $\RCD$ spaces, the first non-zero eigenvalue of any such cross-section satisfies $\lambda_1(N)\ge n-1$. By Obata's rigidity theorem \cite{KettererObata}, the inequality is strict for each such $N$. Indeed, if not, $N$ would be a spherical suspension and we already argued that this is not possible under the present assumptions. The conclusion follows from a standard compactness argument based on the stability of the first eigenvalue of the Laplacian under (measured) Gromov-Hausdorff convergence with lower Ricci curvature bounds and on the compactness of the set of possible cross-sections of blow-downs of $(M,g)$ with respect to the same topology (see, e.g., \cite[Theorem 1.9.4]{AmbrosioHondastab}).
    
    \medskip

\textbf{Step 2.} Let $V_i\to \infty$ be a sequence of volumes and let $E_i\subset M$ be isoperimetric sets with $\vol(E_i)=V_i$. Recall the expression for $r_V$ that was introduced in \eqref{eq:rV}. We claim that with the metrics induced by the rescaled ambient spaces $(M^n,r_{V_i}^{-1}\d)$ the boundaries $\partial E_i$ converge in the Gromov-Hausdorff sense to a cross-section of a blow-down of $M$, up to the extraction of a subsequence. Moreover, the induced metrics on $\partial E_i$ have uniform two-sided sectional curvature bounds.
Hence they converge in $C^{1,\alpha}$ for every $\alpha<1$. The rest of this step will be aimed at establishing the Gromov-Hausdorff convergence with the induced metrics. 
\medskip

We first note that the uniform two-sided bounds for the sectional curvatures of the induced Riemannian metrics on $\partial E_i$ (after rescaling) follow from the quadratic curvature decay of $(M,g)$, the second estimate in \eqref{eq:Hausconv}, the convergence of the second fundamental forms from \cref{prop:secondffstable} and the Gauss equations. 
\smallskip

Up to the extraction of a subsequence that we do not relabel, $(M^n,r_{V_i}^{-1}\d,p)$ converge in the pointed Gromov-Hausdorff sense to a blow-down cone $C(N)$ of $(M,g)$. Moreover, the convergence is $C^{1,\alpha}$ for every $\alpha<1$ away from the vertex of the cone. We claim that $\partial E_i$ with the induced Riemannian metrics converge to the cross-section $(N,g_N)$ in the Gromov-Hausdorff sense.
We note that the $(n-1)$-dimensional measures of $\partial E_i$, i.e., their perimeters, computed with respect to the rescaled metrics as above, converge to $\mathcal{H}^{n-1}(N)$ by \cite[Theorem 4.23]{AntonelliPasqualettoPozzettaSemola1}. Moreover, from the forthcoming \cref{lem:diaminj} we infer that the boundaries $\partial E_i$ with the induced metrics have uniformly bounded diameters (after rescaling) and injectivity radii uniformly bounded away from $0$. Thanks to the discussion above they also have uniform two-sided bounds on the sectional curvature. In particular, up to the extraction of a subsequence that we do not relabel, by Cheeger's precompactness theorem they converge in the Gromov-Hausdorff sense to a Riemannian manifold $(N',g')$ where $g'$ is a $C^{1,\alpha}$ Riemannian metric for every $\alpha<1$, and the convergence is $C^{1,\beta}$ for every $\beta<1$. We claim that $(N',g')$ is isometric to $(N,g_N)$. Thanks to the Hausdorff convergence of $\partial E_i$ to $N$ from \cref{lemma:hausconv} we can find a $1$-Lipschitz surjective map $F:(N',g')\to (N,g_N)$. Moreover, we already argued above that
\begin{equation}
    \mathcal{H}^{n-1}(N')=\lim_{i\to\infty}\mathcal{H}^{n-1}(\partial E_i)=\mathcal{H}^{n-1}(N)\, .
\end{equation}
Above, it is understood that $\mathcal{H}^{n-1}(\partial E_i)$ is computed with respect to the rescaled metric on $M$. The well-known Lipschitz volume rigidity in Riemannian geometry (see e.g. \cite[Lemma 9.1]{BuragoIvanov10} and references therein) guarantees that $F$ is an isometry. Therefore $\partial E_i$ with the induced length distances Gromov-Hausdorff converge to $N$ thus completing the proof of the claim.
\medskip

\textbf{Step 3.} We claim that there exist $\epsilon>0$ and $V_0>0$ such that if $V>V_0$ and $V\in \mathcal{G}_{\epsilon}$, where the set $\mathcal{G}_{\epsilon}$ was introduced in the statement of \cref{lemma:Ricciiso0}, then the first eigenvalue (with respect to volume preserving variations) of the Jacobi operator on the boundary of each isoperimetric set $E_V\subset M$ with $\vol(E_V)=V$ is at least $\varepsilon_MV^{-2/n}$, where $\varepsilon_M$ has been introduced in Step 1 above. In particular, $E_V$ has strictly volume preserving stable boundary. 
The proof of the claim will complete the proof of \cref{thm:strictstab}.

We argue by contradiction, relying on the convergence and stability of Sobolev functions defined along Gromov-Hausdorff converging sequences of manifolds, see \cite{GigliMondinoSavare,Hondaelliptic,AmbrosioHondastab} for the relevant background. If the claim does not hold then there exists a sequence $V_i\to \infty$ and isoperimetric sets $E_i\subset M$ with $\vol(E_i)=V_i$ such that 
\begin{equation}\label{eq:Riccia0}
    \lim_{i\to \infty}V_i^{\frac{2}{n}}\fint_{\partial E_i}\ric(\nu_{\partial E_i},\nu_{\partial E_i})=0,
\end{equation}
and there are smooth functions $f_i:\partial E_i\to\R$ with $f_i$ not identically vanishing such that
\begin{equation}\label{eq:zeroav}
    \int_{\partial E_i}f_i\d P_{\partial E_i}=0\, ,
\end{equation}
and 
\begin{equation}\label{eq:nonstrict}
    \int_{\partial E_i}|\nabla_{\partial E_i}f_i|^2\d P_{\partial E_i}\le \int_{\partial E_i}\left[\ric(\nu_{\partial E_i},\nu_{\partial E_i})+|\sff_{\partial E_i}|^2+\frac{3}{2}\varepsilon_M\right]f_i^2\d P_{\partial E_i}\, ,
\end{equation}
for each $i\in \N$. Note that both \eqref{eq:Riccia0}, \eqref{eq:zeroav}, and \eqref{eq:nonstrict} are scale-invariant. We let $\Sigma_i:=\partial E_i$ endowed with the metric $g_i$ induced by the embedding into $(M^n,r_{V_i}^{-1}\d)$ and the induced Riemannian $(n-1)$-dimensional volume $\vol_i$. By Step 2, $(\Sigma_i,g_i,\vol_i)$ converge in $C^{1,\alpha}$ for every $\alpha<1$ to $(N,g_N,\vol_N)$, up to the extraction of a subsequence that we do not relabel. Here $N$ denotes the cross-section of a blow-down $C(N)$ of $(M,g)$. The conditions \eqref{eq:zeroav} and \eqref{eq:nonstrict} can be rephrased by saying that there exists a sequence of smooth functions $\bar{f}_i:\Sigma_i\to \R$ such that 
\begin{equation}\label{eq:scaledmeans}
    \int_{\Sigma_i}\bar{f}_i\d\vol_i=0\, ,\quad\quad \int_{\Sigma_i}{\bar{f}_i}^2\d \vol_i=1\, ,
\end{equation}
and 
\begin{equation}\label{eq:scalecontra}
    \int_{\Sigma_i}|\nabla \bar{f}_i|^2\d\vol_i\le \int_{\Sigma_i}h_i\bar{f}_i^2\d \vol_i\, ,
\end{equation}
where we set
\begin{equation}\label{eq:defgi}
    h_i:=\ric(\nu_{\Sigma_i},\nu_{\Sigma_i})+|\sff_{\Sigma_i}|^2+\frac{3}{2}\varepsilon_M\, ,
\end{equation}
for each $i\in\N$.

By \cref{prop:secondffstable}
\begin{equation}\label{eq:sffre}
    |\sff_{\Sigma_i}|^2\to (n-1), 
\end{equation}
uniformly as $i\to\infty$. Hence, by \eqref{eq:Riccia0} and the quadratic curvature decay \eqref{eq:QCD}, $h_i\to (n-1)+\frac{3}{2}\varepsilon_M$ in $L^q$ as $i\to\infty$ for every $1\le q<\infty$. By the quadratic curvature decay assumptions again and \eqref{eq:sffre}, the functions $h_i$ are uniformly bounded with respect to $i$. Hence, by \eqref{eq:scaledmeans} and \eqref{eq:scalecontra}, the functions $\bar{f}_i$ have uniformly bounded $H^{1,2}$ norms. By Step 2, a $(p,2)$-Sobolev inequality with uniform constants independent of $i\in\N$ holds along the sequence $(\Sigma_i,g_i)$, for every $p>1$, see also \cref{lem:diaminj} below. Hence for any $p>1$, up to the extraction of a subsequence that we do not relabel, the functions $\bar{f}_i$ converge in $L^p$ to a function $\bar{f}:N\to \R$. In particular, if we choose $p>2$, then by \eqref{eq:scaledmeans} we get that 
\begin{equation}
   \int_N\bar{f}\d\vol_N=0\, ,\quad \int_N\bar{f}^2\d\vol_N=1\, . 
\end{equation}
 Moreover, the combination of the $L^q$ convergence of $h_i\to (n-1)$ with the $L^p$ convergence of $\bar{f}_i\to \bar{f}$ yields that 
 \begin{equation}
     \int_{\Sigma_i}h_i\bar{f}_i^2\d \vol_i\to (n-1)\int_N\bar{f}^2\d\vol_N=n-1\, \quad \text{as $i\to\infty$}\, .
 \end{equation}
Taking into account the lower semicontinuity of the Dirichlet energy under $L^2$ convergence we get from \eqref{eq:scalecontra} that
\begin{equation}
    \int_N|\nabla \bar{f}|^2\d\vol_N\le n-1+\frac{3}{2}\varepsilon_M\, .
\end{equation}
This is a contradiction to the uniform spectral gap estimate that was obtained in Step 1. 
\end{proof}

\begin{remark}
    Regarding the argument in Step 1 above, we note that in the present setting, where two-sided uniform sectional curvature bounds are in force, the set of cross-sections is compact in the $C^{1,\alpha}$ topology and it is much more elementary to check that the spectrum of the Laplacian is stable in this case, see \cite[Chapter 10]{Petersenbook} for the relevant background.
\end{remark}

\cref{prop:secondffstable} is a consequence of the effective regularity theory for (almost) minimizers of the perimeter under two-sided curvature bounds that is discussed in \cref{appA} and \cref{lemma:hausconv}.

\begin{proof}[Proof of \cref{prop:secondffstable}]
    Arguing by contradiction, it is enough to prove the following statement. Let $V_i\to\infty$, and let $\widetilde E_i$ be isoperimetric sets with volume $V_i$. Then, up to subsequences, 
    \begin{equation}\label{eqn:ControlSecondFormSequence}
 \lim_{i\to\infty}\sup_{x \in \partial \widetilde E_i} \left| V_i^{\frac2n} \, |\sff_{\partial \widetilde E_i}|^2(x) - (n-1)(\omega_n {\rm AVR}(M))^{\frac2n} \right| =0\, .
\end{equation}

Denote $r_i:=(\mathrm{AVR}(M)\omega_n)^{-\frac 1n}V_i^{\frac 1n}$, and fix $0<\alpha'<\alpha<1$. Let $E_i$ denote the set $\widetilde E_i$ considered in the rescaled manifold $M_{i}:=(M^n,r_i^{-1}\mathrm{d}_g,p)$. Up to subsequences, let $(C(N),o)$ be the asymptotic cone that is the limit of $M_i$ in the pointed Gromov--Hausdorff sense, and in the $C^{1,\alpha}$-sense outside the tip $o$. Let $g^C$ be the limit metric metric on $C(N)$, and let $g_i$ denote the rescaled metric on $M_i$. We aim at proving the following assertion.
\smallskip

\textbf{Claim}. There are $\rho,\overline{C}>0$ such that for every $p\in \partial B_1^{C(N)}(o)$ there exists a chart $\varphi:B_\rho(0^n)\to C(N)$ centered at $p$ such that for every $\partial E_i\ni p_i\to p$ the following holds up to subsequences in $i$. There exist open sets $B_{\rho/2}(p)\subset \varphi(B_\rho(0^n))\subset U_i\subset C(N)$, and $B_{\rho/2}^{M_i}(p_i)\subset V_i\subset M_{i}$, and diffeomorphisms $F_i:U_i\to V_i$ with $F_i(p)=p_i$ for which:
    \begin{enumerate}
        \item It holds 
        \begin{equation}\label{eqn:MetricC1alphaConvergence}
        \|\varphi^*g^C-(F_i\circ\varphi)^*g_i\|_{C^{1,\alpha}(B_{\rho/2}(0^n))} \to 0\, .
        \end{equation}
        \item There exist $u,u_i:B_{\rho/2}(0^{n-1})\to \mathbb R$ such that
        \begin{equation}\label{eqn:GraphU}
            \varphi^{-1}(\partial B_1^{C(N)}(o))\cap B_{\rho/2}(0^n) = \{(x,u(x)):x\in B_{\rho/2}(0^{n-1})\}\cap B_{\rho/2}(0^n)\, ,
        \end{equation}
        \begin{equation}\label{eqn:AllardGraph}
            (F_i\circ\varphi)^{-1}(\partial E_i)\cap B_{\rho/2}(0^n) = \{(x,u_i(x)):x\in B_{\rho/2}(0^{n-1})\}\cap B_{\rho/2}(0^n)\, ,
        \end{equation}
        \begin{equation}\label{eqn:AllardBound}
            \|u_i\|_{C^{2,\alpha}(B_{\rho/4}(0^{n-1}))} \leq \overline C\, \quad \text{for every $i\in\N$}\, ,
        \end{equation}
        and
        \begin{equation}\label{eqn:Convergence}
            u_i\to u \qquad \text{in $C^{2,\alpha'}(B_{\rho/8}(0^{n-1}))$}\, .
        \end{equation}
    \end{enumerate}

    In particular, 
    \begin{equation}\label{eqn:SFFConvergenceFinally}
    \lim_{i\to\infty}\sup_{x\in\partial E_i}||\sff_{\partial E_i}|(x)-(n-1)| = 0\, .
    \end{equation}
    \smallskip

The proof the present proposition will be completed after the proof of the \textbf{Claim}. Indeed, \eqref{eqn:SFFConvergenceFinally} is equivalent to \eqref{eqn:ControlSecondFormSequence} after scaling.
    \smallskip
    
    In order to prove the \textbf{Claim} we rely on the results in \cref{appA}. Item (1) comes from the $C^{1,\alpha}$-convergence of $M_i$ to $C(N)$ outside the tip $o$. Moreover, \eqref{eqn:AllardGraph}, and \eqref{eqn:AllardBound} in Item (2) come from (the proof of) \cref{thm:SecondeFormeConvergono}, and Allard's result \cref{C1gamma}, see in particular \eqref{eqn:C2alphabound}. Recall that \cref{thm:SecondeFormeConvergono} can be applied in the setting of the present \cref{prop:secondffstable} due to \cref{rem:ApplicationToIsop}.  Finally \eqref{eqn:AllardBound} in item (2) implies precompactness in $C^{2,\alpha'}(B_{\rho/8}(0^{n-1}))$, so that $u_i\to \tilde u$ up to subsequences in $C^{2,\alpha'}(B_{\rho/8}(0^{n-1}))$. Since $\partial E_i\to \partial B_1^{C(N)}(o)$ in the Hausdorff sense, we infer that $\tilde u=u$, and thus \eqref{eqn:Convergence} is proved. Finally, \eqref{eqn:SFFConvergenceFinally} is a direct consequence of \eqref{eqn:MetricC1alphaConvergence}, and \eqref{eqn:Convergence}.
\end{proof}

The next Lemma is a consequence of the results in \cref{appA}, and it was already used in the proof of \cref{thm:strictstab} above.

\begin{lemma}\label{lem:diaminj}
    Let $(M^n,g)$ be a smooth complete Riemannian manifold with $\mathrm{Ric}\ge 0$, Euclidean volume growth \eqref{eq:EVG}, and quadratic Riemann curvature decay \eqref{eq:QCD}. There exists $V_0,C>0$ such that the following holds.
    Let $E_V$ be an isoperimetric set of volume $V>V_0$, and let $\Sigma:=\partial E_V$. Let $g_\Sigma$ be the Riemannian metric induced by $g$ on $\Sigma$. Then 
    \begin{equation}\label{eqn:DiamEInj}
    \mathrm{diam}(\Sigma,g_\Sigma) \leq CV^{1/n}\, , \qquad \mathrm{inj}(\Sigma,g_\Sigma) \geq CV^{1/n}\, .
    \end{equation}
\end{lemma}

\begin{proof}
    It is enough to prove the following \textbf{Claim}, and then argue by contradiction.
\smallskip

    \textbf{Claim}. Let $V_i\to\infty$, and let $\widetilde E_i\subset M$ be an isoperimetric set with volume $V_i$. Denote $r_i:=(\mathrm{AVR}(M)\omega_n)^{-\frac 1n}V_i^{\frac 1n}$. Let $E_i$ denote the set $\widetilde E_i$ considered in the rescaled manifold $M_{i}:=(M^n,\d_i:=r_i^{-1}\mathrm{d}_g,p)$. Set $\Sigma_i:=\partial E_i$, and denote by $g_i$, $\mathrm{d}_i$ the metric and distance on $M_i$, and by $g_{\Sigma_i}$ the induced metric on $\Sigma_i$. Then 
    \[
    \limsup_{i\to \infty} \mathrm{diam}(\Sigma_i,g_{\Sigma_i})<\infty\, , \qquad \liminf_{i\to\infty}\inj(\Sigma_i,g_{\Sigma_i}) >0\, .
    \]

\smallskip

    By \eqref{eqn:MetricC1alphaConvergence}, \eqref{eqn:AllardGraph}, and \eqref{eqn:Convergence} in the proof of \cref{prop:secondffstable}, the intrinsic and the restricted distances on $\Sigma_i$ are locally equivalent. I.e., there exist constants $\rho,C>0$ such that, for all $i$ large enough (up to subsequences) we have
    \begin{equation}\label{eqn:LocalUniform}
    \forall p\in\Sigma_i,\quad \forall q, r\in B_{\rho}^{M_i}(p)\cap\Sigma_i, \quad |\mathrm{d}_{\Sigma_i}(q,r)-\mathrm{d}_i(q,r)|\leq C\mathrm{d}_i(q,r)\, .
    \end{equation}
    Thus, joining \eqref{eqn:LocalUniform} with the density estimates in \cite[Corollary 4.15]{AntonelliPasqualettoPozzettaSemola1}, we get that, possibly taking a smaller $\rho$, there exists a constant $\xi>0$ such that, for all $i$ large enough,
    \begin{equation}\label{eqn:UniformLowerBoundVol}
    \forall p\in \Sigma_i,\, \vol_{\Sigma_i}(B_\rho^{\Sigma_i}(p))\geq \xi>0\, .
    \end{equation}
    Above $B_\rho^{\Sigma_i}$ denotes the ball in the metric $g_{\Sigma_i}$. By \eqref{eqn:SFFConvergenceFinally}, the quadratic Riemann curvature decay \eqref{eq:QCD}, and the Gauss equation, there exists a constant $K>0$ such that, for $i$ large enough, 
    \begin{equation}\label{eqn:UniformUpperBoundSec}
    |\mathrm{Sec}_{\Sigma_i}|\leq K\, .
    \end{equation}
    By \eqref{eqn:UniformLowerBoundVol}, \eqref{eqn:UniformUpperBoundSec}, and \cite[Lemma 11.4.9]{Petersenbook} we obtain that 
\begin{equation}\label{eqn:InfoLiminfinj}
\liminf_{i\to\infty}\mathrm{inj}(\Sigma_i,g_{\Sigma_i})>0\, .
    \end{equation}
    Finally, observe $|\Sigma_i|\to |\partial B_1^{C(N)}(o)|$ by \cite[Theorem 4.23]{AntonelliPasqualettoPozzettaSemola1}, where $B_1^{C(N)}(o)$ is the unit ball centered at the tip of an asymptotic cone $M_i$ is converging to, up to subsequences. Hence $|\Sigma_i|$ is uniformly bounded above. Joining the latter information with \eqref{eqn:InfoLiminfinj}, and \eqref{eqn:UniformUpperBoundSec}, a standard covering argument using Bishop--Gromov comparison theorem gives 
    \[
\limsup_{i\to\infty}\mathrm{diam}(\Sigma_i,g_{\Sigma_i})<\infty\, .
    \]
    Hence the \textbf{Claim} is proved, and the proof is concluded.
\end{proof}

As already mentioned, controlling the Ricci curvature term of the stability operator is the most delicate part of the argument for proving strict stability. One of the reasons is that, contrary to the other terms of the stability operator, the Ricci curvature term does not pass to the limit under blow-downs in general, and a careful volume selection argument is needed.

We note that a quadratic \emph{lower} curvature decay in \eqref{eq:QCD}, combined with \eqref{eq:EVG} is sufficient to guarantee that the Ricci curvature of $(M^n,g)$ weakly converges to $0$ in the almost radial directions when we consider any blow-down, see \cite[Theorem 1.1]{PetruninLebedeva}. However, even under the stronger two-sided quadratic curvature decay assumption, this weak convergence does not localize to (all) boundaries of large isoperimetric sets.
\smallskip

To illustrate the idea of the proof of \cref{lemma:Ricciiso0} it is helpful to introduce the notion of isoperimetric profile of $(M^n,g)$. We define the isoperimetric profile $I_M=I:[0,\infty)\to [0,\infty)$ by 
\begin{equation}
I(v):=\inf\{P(E)\, :\, E\subset M\, ,\,  \vol(E)=v\}\, .
\end{equation}
and let $J_M=J:=I^{\frac{n}{n-1}}$ be the normalized isoperimetric profile.
In \cite{AntonelliPasqualettoPozzettaSemola1} it was shown that if $(M^n,g)$ is complete and satisfies $\ric\ge 0$, then $J$ is a concave function, see also the earlier \cite{MondinoNardulli} where the same result was obtained under some additional assumptions. Moreover, if ${\rm AVR}(M)>0$, then 
\begin{equation}
J\sim J_{C(N)}
\end{equation}
for large volumes, where $J_{C(N)}$ denotes the normalized isoperimetric profile of one of the blow-downs of $(M^n,g)$. Note that $J_{C(N)}$ is an affine function whose slope depends only on ${\rm AVR}(M)$ (equivalently, on $\vol(N)$, which is independent of the blow-down), see \cite{MorganRitore02,AntonelliPasqualettoPozzettaSemola1}. 

For the proof of \cref{lemma:Ricciiso0} we will reverse the usual argument for proving the concavity of $J$ via the second variation formula, see \cite{BavardPansu, BrayPhD,Bayle04,AntonelliPasqualettoPozzettaSemola1}. The main idea is that $J$ is concave and asymptotically affine. Hence its second derivative converges to $0$ with a rate in an integral $L^1$-sense as $V\to\infty$. In particular, it converges to $0$ with the same rate on sets of volumes of scale invariantly larger and larger measure. For those volumes, we shall be able to control the Ricci curvature term of the stability operator. 

The proof of the forthcoming \cref{lemma:maxRicci} will require an effective version of the idea discussed above. We prefer to prove the weaker \cref{lemma:Ricciiso0} separately in order to introduced the first main idea of the proof in a simplified case.

\begin{proof}[Proof of \cref{lemma:Ricciiso0}]
By \cite[Theorem 1.1]{AntonelliPasqualettoPozzettaSemola1}, letting $J:(0,\infty)\to (0,\infty)$ be defined as $J:=I^{\frac{n}{n-1}}$ as above, $J$ is a concave function. By \cite[Corollary 3.6]{AntonelliBrueFogagnoloPozzetta} the right-derivative $J_+'$ (which is well-defined at every point by concavity) has a finite limit as $v\to \infty$. Denoting by $J''$ the second derivative, which is well-defined $\mathcal{L}^1$-a.e., we get from elementary convex analysis
that
\begin{equation}\label{eq:intJ''}
    \int_r^{2r}J''(v) \d v \to 0\, ,\quad \text{as $r\to \infty$}\, .
\end{equation}
For every $\delta>0$ we let $\mathcal{F}_{\delta}\subset (0,\infty)$ be the set of points $v\in (0,\infty)$ such that $J$ is twice differentiable at $v$ and $0\ge J''(v)v\ge -\delta$. Taking into account \eqref{eq:intJ''} and the nonpositivity of $J''$ it is elementary to infer that 
\begin{equation}
    \lim_{r\to \infty}\frac{\mathcal{L}^1(\mathcal{F}_{\delta}\cap (r,2r))}{r}=1\, ,
\end{equation}
for every $\delta>0$.
\medskip

Fix $\delta>0$ and let us consider any volume $V\in \mathcal{F}_{\delta}$. Let $E\subset M$ be any isoperimetric region in $M$ with $\vol(E)=V$, if there exist any.
Let $E_t$ be the $t$-enlargement of $E$ for all $t$ in a small neighborhood of $0$, i.e., $E_t:=\{x\in M\, :\, \d(x,E)<t\}$ if $t>0$ and $E_t:=\{x\in M\, : \, \d(x,M\setminus E)>-t\}$ for $t\le 0$. A standard argument in geometric measure theory based on the first variation formula for the perimeter shows that it is possible to find a continuous function $\tau$ defined in a neighborhood $(V-\mu,V+\mu)\ni V$ with $\tau(V)=0$ and such that $\vol(E_{\tau(v)})=v$ for each $v\in (V-\mu,V+\mu)$. We note that, by the very definition of the isoperimetric profile, it holds
\begin{equation}\label{eq:compar}
   \tilde{I}_V(v):= P(E_{\tau(v)})\ge I(v)\, ,
\end{equation}
for every $v\in (V-\mu,V+\mu)$. Moreover, setting $\tilde{J}_V:=\tilde{I}_V^{\frac{n}{n-1}}$, we have
\begin{align}
\label{eq:first lineIIv}    \frac{\d^2}{\d v^2}|_{v=V}\tilde{J}_V=&\frac{n}{n-1}P(E)^{\frac{2-n}{n-1}}\left[\frac{H^2}{n-1}-\fint_{\partial E}\left(|\sff_{\partial E}|^2+\ric(\nu_{\partial E},\nu_{\partial E})\right)\right]\\
\label{eq:secondline IIv}    \le &  -\frac{n}{n-1}P(E)^{\frac{2-n}{n-1}}\fint_{\partial E}\ric(\nu_{\partial E},\nu_{\partial E})\, .
\end{align}
Above, in order to pass from \eqref{eq:first lineIIv} to \eqref{eq:secondline IIv} we applied the Cauchy-Schwartz inequality to the $(n-1)$ eigenvalues of $\sff_{\partial E}$. By \eqref{eq:compar} it readily follows that $\tilde{J}_V\ge J$ on $(V-\mu,V+\mu)$ and $\tilde{J}_V(V)=J(V)$. Hence, since $V\in \mathcal{F}_{\delta}$, we get
\begin{equation}\label{eq:estimate''}
    \frac{n}{n-1}P(E)^{\frac{2-n}{n-1}}\fint_{\partial E}\ric(\nu_{\partial E},\nu_{\partial E})\le \frac{\delta}{V}\, .
\end{equation}
If $\partial E$ is smooth, the second derivative of $\tilde J_V$ above makes sense in the classical sense. Otherwise, the present argument can be adapted to infer \eqref{eq:estimate''} even when $E$ has singular boundary, by using the fact that the codimension of the singular part of $\partial E$ is at least 8, and standard arguments in geometric measure theory, compare with \cite{Bayle04}. We note that there exists a constant $C>1$ depending only on the asymptotic volume ratio of $M$ such that
\begin{equation}
  \frac{1}{C}V^{\frac{n-1}{n}} \le  P(E)\le C V^{\frac{n-1}{n}}\, ,
\end{equation}
for any isoperimetric set $E\subset M$ with $\vol(E)=V$.
Therefore, the statement follows from \eqref{eq:estimate''} up to choosing $\delta<\delta(\epsilon)$ sufficiently small and setting $\mathcal{G}_{\epsilon}:=\mathcal{F}_{\delta(\epsilon)}$\, .
\end{proof}

\section{Uniqueness for many volumes}

The goal of this section is to prove the uniqueness part of \cref{thm:mainthm}. We state below the precise result for the ease of readers.

\begin{thm}\label{thm:unique}
    Let $(M^n,g)$ be a smooth complete Riemannian manifold with $\mathrm{Ric}\ge 0$, Euclidean volume growth \eqref{eq:EVG}, and quadratic Riemann curvature decay \eqref{eq:QCD}. Assume that $(M^n,g)$ is not isometric to $\mathbb{R}^n$. Then there exists a measurable set $\mathcal{G}\subset (0,\infty)$ such that 
    \begin{equation}
        \lim_{r\to \infty}\frac{\mathcal{L}^1(\mathcal{G}\cap (r,2r))}{r}=1\, ,
    \end{equation}
    and for any $V\in \mathcal{G}$ the isoperimetric set $E_V\subset M$ of volume $V$ is unique. 
\end{thm}

The key idea will be to make effective the broad principle according to which strictly stable implies locally uniquely minimizing.
As mentioned in the introduction, the proof of \cref{thm:unique} is divided in a few steps. A first technical tool is the fact that the reach of isoperimetric boundaries of large volumes is uniformly bounded below in a scale-invariant form. We recall that for a bounded open subset $\Omega\subset M$ with smooth boundary in a Riemannian manifold $(M,g)$, the reach of $\partial\Omega$ is the supremum of all positive numbers $r$ such that for every point $x_0$ in $\{x \in M \st \dist(x,\partial \Omega)<r\}$ there exists a unique $y_0 \in \partial \Omega$ such that $\dist(x_0,y_0) = \dist(x_0,\partial \Omega)$. In particular the signed distance function from $\Omega$ is smooth in the open tubular neighborhood of $\partial \Omega$ having size equal to the reach of $\partial\Omega$.

\begin{lemma}\label{lem:PositiveReach}
     Let $(M^n,g)$ be a smooth complete Riemannian manifold with $\mathrm{Ric}\ge 0$, Euclidean volume growth \eqref{eq:EVG}, and quadratic Riemann curvature decay \eqref{eq:QCD}. Then there exists $\overline{V}, \overline{\rho}>0$ such that for every $V\ge \overline{V}$ the reach of the boundary of every isoperimetric set $E_V\subset M$ of volume $\vol(E_V)=V$ is bounded from below by $\overline{\rho}V^{\frac1n}$.
\end{lemma}

\begin{proof}
After scaling, for a constant $\mathcal{C}$ depending on $M$, we can prove the claim for isoperimetric sets $E$ having volume $\mathcal{C}{\rm AVR}(M)\omega_n$ that are $(\Lambda,1)$-minimizers (\cref{def:LambdaMin}) for some fixed $\Lambda>0$ in manifolds with a uniform bound on $|{\rm Riem}|$, and such that \eqref{eqn:C1Control}, and \eqref{eqn:DensityControlPerimeter} hold for a suitable choice of local chart at boundary points of $E$ (see \cref{rem:ApplicationToIsop}, \cref{prop:AdaptEpsRegularity}, and \cref{thm:HarmonicRadius}). In particular \cref{C1gamma} can be applied at boundary points of $E$ and \eqref{eqn:GraphConditions} implies that \eqref{eqn:2SidedWell} is satisfied for some $\eps_0$ independent of the specific isoperimetric set. Taking into account also \cref{prop:secondffstable}, the conclusion follows by \cref{prop:reachbound}.
\end{proof}

The first main step for the proof of \cref{thm:unique} amounts to prove that for every sufficiently large volume away from a countable set, any two isoperimetric sets with that volume can be written one as a normal graph over the boundary of the other, with suitable scale-invariant bounds on the graphing function. The set of \textit{bad} volumes consists of the set of non-differentiability points of the isoperimetric profile $I$ of $(M,g)$. The precise statement is as follows. 

\begin{prop}\label{prop:graphfunction}
Let $(M^n,g)$ be a smooth complete Riemannian manifold with $\mathrm{Ric}\ge 0$, Euclidean volume growth \eqref{eq:EVG}, and quadratic Riemann curvature decay \eqref{eq:QCD}. Then there exists a constant $\mathcal{C}>0$ such that the following holds.

For any $\epsilon>0$ there exists $V_0=V_0(\epsilon)>0$ such that for every $V>V_0$ away from a countable set the following holds. If $E_V,E'_V\subset M$ are two distinct isoperimetric regions with $\vol(E_V)=\vol(E'_V)=V$, then $\partial E'_V$ is the normal graph over $\partial E_V$ of a function $w:\partial E_V\to \R$ satisfying:
\begin{equation}\label{eqn:FirstThing}
    \norm{w}_{\infty}\le \epsilon V^{\frac{1}{n}}\, ,\quad \norm{\nabla w}_{\infty}\le \varepsilon\, ,
\end{equation}
    \begin{equation}\label{eqn:SecondThing}
        \norm{w}_{\infty}\le \mathcal{C}V^{-\frac{n-1}{2n}}\norm{w}_{L^2(\partial E_V)}\, ,
    \end{equation}
    \begin{equation}\label{eqn:ThirdThing}
        \norm{\nabla w}_{L^2(\partial E_V)}\le \mathcal{C}V^{-\frac 1n}\norm{w}_{L^2(\partial E_V)}\, .
    \end{equation}
\end{prop}

In the second step of the proof, we are going to effectively control the average of the Ricci curvature in \emph{normal} directions on tubular neighborhoods of any size for isoperimetric sets of most volumes. The precise statement follows:

\begin{prop}\label{lemma:maxRicci}
  Let $(M^n,g)$ be a smooth complete Riemannian manifold with $\mathrm{Ric}\ge 0$, Euclidean volume growth \eqref{eq:EVG}, and quadratic Riemann curvature decay \eqref{eq:QCD}. Then for every $\epsilon>0$ there exists a set $\mathcal{G}_{\epsilon}\subset (0,\infty)$ such that 
   \begin{equation}\label{eq:DensityOneVolumiRicciBasso}
        \lim_{r\to \infty}\frac{\mathcal{L}^1(\mathcal{G}_{\epsilon}\cap (r,2r))}{r}=1\, ,
    \end{equation}
 and for every volume $V\in \mathcal{G}_{\epsilon}$ and every isoperimetric set $E_V\subset M$ with volume $V$ it holds
    \begin{equation}\label{eq:stimamassimale}
   \sup_{0<r<C_0V^{\frac{1}{n}}} V^{\frac{2}{n}} \fint_{B_r(\partial E_{V})} \ric(\nabla \d_{\partial E_V},\nabla \d_{\partial E_V}) \d \vol\le \epsilon\,  .
    \end{equation}
\end{prop}

In order to complete the proof of \cref{thm:unique} we are going to perform a Taylor expansion to compare the perimeter of two distinct isoperimetric sets $E_V$ and $E_V'$ with the same volume $V$. The identity $\mathrm{Per}(E_V)=\mathrm{Per}(E_V')$ will result in a contradiction provided that the volume $V$ is suitably chosen. In the next section we discuss the details by taking for granted \cref{prop:graphfunction} and \cref{lemma:maxRicci}, whose proofs are deferred to a subsequent section.

\subsection{Proof of \cref{thm:unique} given \cref{prop:graphfunction} and \cref{lemma:maxRicci}}

Let $\eps>0$ and let $\mathcal{G}_\eps$ be given by \cref{lemma:maxRicci}. Without loss of generality we can assume that the conclusion of \cref{thm:strictstab} holds for all volumes $V\in \mathcal{G}_{\epsilon}$. We claim that if $\eps$ is small enough, then the theorem follows taking $\mathcal{G}=\mathcal{G}_\eps \cap  (v_0,\infty)$ for some $v_0>0$.
\smallskip

We will denote by $\mathcal{C}>0$ any constant depending on $M$ only, independent of $\eps$, that may change from line to line.
Up to removing a countable (and hence negligible) set from $\mathcal{G}_\eps$, we can assume that the isoperimetric profile is differentiable at any volume in $\mathcal{G}_\eps$ and that \cref{prop:graphfunction} can be applied at any volume in $\mathcal{G}_\eps \cap (V_0(\eps), \infty)$.
The strategy will be to argue by contradiction. 
Assume that there exists a sequence $V_i \in \mathcal{G}_\eps$ such that $V_i\nearrow\infty$ and such that there exist two distinct isoperimetric sets $E_{V_i}, E_{V_i}'$ of volume $V_i$. There will be three main steps. In \textbf{Step 1} we are going to rely on \cref{prop:graphfunction} to infer that $\partial E_{V_i}'$ is a graph over $\partial E_V$, and to use the assumption that $\vol(E_{V_i})=\vol(E_{V_i}')=V_i$ to obtain an effective estimate on the average of the graphing function. In \textbf{Step 2} we are going to perform a Taylor expansion of the perimeter functional in a neighborhood of $\partial E_{V_i}$ in terms of the stability operator and an error term depending on the ambient Ricci curvature. In \textbf{Step 3} we will conclude with the help of \cref{lemma:maxRicci}, whose aim is exactly to control the potentially large Ricci curvature term.
\medskip

{\textbf{Step 1.}} Denote $r_i:=(\mathrm{AVR}(M)\omega_n)^{-\frac 1n}V_i^{\frac 1n}$, and $M_{V_i}:=(M^n,r_i^{-1}\mathrm{d}_g,p)$. Let $g_i$ be the rescaled metric on $M_{V_i}$. After rescaling and up to extracting a subsequence, we can assume that $M_{V_i}$ converges in Gromov--Hausdorff sense to an asymptotic cone $C(N)$ with $C^{1,\alpha}$-cross section $(N, g_N)$ and the sets $E_i, E_i'$ corresponding to $E_{V_i}, E_{V_i}'$, respectively, converge to the unit ball of $C(N)$ centered at the tip strongly in $L^1$. For large $i$, by \cref{prop:graphfunction} there exists a sequence of functions $w_i \in C^\infty(\partial E_i)$ that parametrize $\partial E_i'$ over $\partial E_i$ as a normal graph, i.e.,
\[
\partial E_i' = \left\{\exp_x(w_i(x) \nu_{E_i}(x)) \st x \in \partial E_i \right\}\, .
\]

We claim that for any $i\in \N$ sufficiently large it holds
\begin{equation}\label{eq:AlmostZeroMean}
    \left|\int_{\partial E_i} w_i - \frac12 H_{\partial E_i} \int_{\partial E_i} w_i^2 \right| \le \eps \,\mathcal{C} \int_{\partial E_i} w_i^2\, .
\end{equation}
The effect of \eqref{eq:AlmostZeroMean} is to tell that the graphing functions $w_i$ have almost zero average on $\partial E_i$, in a very quantitative way. This will be key later when it comes to reaching a contradiction to the strict volume preserving stability of $\partial E_i$.
\smallskip

From \cref{prop:reachbound}, we know that the reach of $\partial E_i$ is bounded below by a constant $\overline{\rho}>0$ independent of $i$. Hence, for any large $i$, we can fix an open neighborhood $U_i$ of $\partial E_i$ such that $\partial E_i' \subset \{ x\st \dist(x,\partial E_i) <\overline{\rho}/2 \} \Subset U_i$, and such that the metric $g_i$ on $U_i$ can be written in Fermi coordinates with respect to $\partial E_i$. Namely,
\begin{equation}
(U_i, g_i) \simeq ((-\overline{\rho}, \overline{\rho}) \times \partial E_i , \d z^2 + g_z)\, ,
\end{equation}
where $g_z$ is the pull-back on $\partial E_i$ via exponential map of the induced metric on the parallel hypersurface $\left\{\exp_x(z \nu_{E_i}(x)) \st x \in \partial E_i \right\}$, for $z \in (-\overline{\rho}, \overline{\rho})$. On the manifold $((-\overline{\rho}, \overline{\rho}) \times \partial E_i , \d z^2 + g_z)$ a normal graph over $\partial E_i$ defined by a function $u \in C^\infty(\partial E_i)$ with $|u|<\overline{\rho}$ is a set of the form $S=\{ ( u(x),x) \st x \in \partial E_i\}$, and the upward unit normal of $S$ can be written as the vector field
\begin{equation}\label{eq:NormaleFermiCoordinates}
    \frac{\partial_z - \nabla^u u}{\left( 1+ |\nabla^u u |^2_u \right)^{\frac12}}\, ,
\end{equation}
where $\nabla^u$ and $|\cdot|_u$ denote gradient and norm with respect to the metric $g_u$ on $\partial E_i$ given by $(g_u)_p := (g_{u(p)})_p$ at any $p \in \partial E_i$. Also, integration with respect to the volume form on the graph $S$ of $u$ can be written in terms of the volume form with respect to $g_u$ as
\begin{equation}\label{eq:FormaVolumeFermiCoordinates}
\int f(z,x) \de \vol_S = \int f(u(x), x) \left( 1+ |\nabla^u u |^2_u \right)^{\frac12} \de\vol_{g_u}\, ,
\end{equation}
for any continuous function $f$ on $S$.

Recall that $\partial_z \sqrt{\det g_z} = -\sqrt{\det g_z} \, \Delta \dist^s$ in any local chart on $\partial E_i$ and $H_{\partial E_i} (x)= \Delta \dist^s(0,x)$, where $\dist^s$ is the signed distance function from $E_i$ ($\dist^s$ is defined to be positive outside $E_i$, and negative inside $E_i$). Then for any $x \in \partial E_i$ working in a local chart on $\partial E_i$ such that $\sqrt{\det g_0}(x)=1$, we can expand
\begin{equation}
\begin{split}
    \sqrt{\det g_z} (z,x) &= 1 - \int_0^z \sqrt{\det g_z}(s,x) \, \Delta \dist^s(s,x) \de s  \\
    &= 1- H_{\partial E_i}z -\int_0^z \int_0^s  |\nabla^2 \dist^s|^2(t,x) + \ric_{(t,x)}(\partial_z, \partial_z)  \de t \de s +\\
    &\qquad - \int_0^z \int_0^s \left( \sqrt{\det g_z}(s,x) - 1\right) \, \Delta \dist^s(s,x) \de s\, .
\end{split}
\end{equation}
Note that $|\sff_{\partial E_i}| , H_{\partial E_i}$ are uniformly bounded with respect to $i$ by \cref{prop:secondffstable}. Exploiting the uniform scale invariant Riemann curvature bounds and the Riccati equation for the evolution of $\Delta \dist^s(s,x)$ and $\nabla^2 \dist^s(s,x)$ with respect to $s$ (see also the proof of \cref{lem:C1K} below), it follows that $|\nabla^2 \dist^s|^2(t,x) + \ric_{(z,x)}(\partial_z, \partial_z) + |\Delta \dist^s(t,x)|\le \mathcal{C}$ for any $x \in \partial E_i$ and any $|t|<\overline{\rho}$. Hence $|\sqrt{\det g_z}(z,x) - 1| \le \mathcal{C} |z|$ and
\begin{equation}
    \left| \sqrt{\det g_z} (z,x) - 1+ H_{\partial E_i}z  \right| \le \mathcal{C} z^2\, .
\end{equation}
Thus
\begin{equation}\label{eq:VolGzVolG0}
    \left|\frac{\de \vol_{g_z}}{\de \vol_{g_0}} - (1- H_{\partial E_i}z )  \right| \le \mathcal{C} z^2\, ,
\end{equation}
for any $z \in (-\overline{\rho}, \overline{\rho})$, where $\frac{\de \vol_{g_z}}{\de \vol_{g_0}}$ is the Radon--Nikod\'ym derivative of $\vol_{g_z}$ with respect to $\vol_{g_0}$.
\medskip

We proceed with the proof of the claimed \eqref{eq:AlmostZeroMean}.
Let $i$ be fixed and let $\{E_t\}_{t \in [0,1]}$ be the one-parameter family of sets having boundary
\begin{equation}
\left\{\exp_x(t w_i(x) \nu_{E_i}(x)) \st x \in \partial E_i \right\}\, ,
\end{equation}
such that $E_0=E_i$, $E_1=E_i'$. Denote $V(t):= \vol_i(E_t)$. Exploiting \eqref{eq:NormaleFermiCoordinates} and \eqref{eq:FormaVolumeFermiCoordinates} with $u=tw_i$ we can compute
\begin{equation}
\begin{split}
    V'(t) 
    &= \int_{\partial E_i} g_i\left( -\frac{\partial_z - t \nabla^{tw_i}w_i}{\left( 1+ t^2|\nabla^{tw_i}w_i|^2_{tw_i} \right)^{\frac12}} , w_i \partial_z\right) \left( 1+ t^2|\nabla^{tw_i}w_i|^2_{tw_i} \right)^{\frac12} \de \vol_{g_{tw_i}}\\
    &= -\int_{\partial E_i} w_i \de \vol_{g_{tw_i}}\, ,
\end{split}
\end{equation}
where $g_i$ here equals $\d z^2 + g_z$ and denotes the metric tensor on $M_{V_i}$ in the Fermi coordinates chart. Hence
\begin{equation}
\begin{split}
    0 &= |E_i|-|E_i'| = -\int_0^1 V'(t) \de t = \int_0^1 \int_{\partial E_i} w_i \de \vol_{g_{tw_i}} \de t  \\
    &= \int_0^1 \int_{\partial E_i} w_i \left(\frac{\de \vol_{g_{tw_i}}}{\de \vol_{g_0}} - (1 - H_{\partial E_i} t w_i) + (1 - H_{\partial E_i} t w_i) \right)\de \vol_{g_0} \de t 
    \\
    &= \int_{\partial E_i} w_i \de \vol_{g_0}- \frac12 H_{\partial E_i} \int_{\partial E_i} w_i^2\de \vol_{g_0} + \int_0^1 \int_{\partial E_i} w_i \left(\frac{\de \vol_{g_{tw_i}}}{\de \vol_{g_0}} - (1 - H_{\partial E_i} t w_i)\right)\de \vol_{g_0} \de t\, .
\end{split}
\end{equation}
Therefore
\begin{equation}
\begin{split}
    \left|\int_{\partial E_i} w_i - \frac12 H_{\partial E_i} \int_{\partial E_i} w_i^2 \right| &\le \left|
    \int_0^1 \int_{\partial E_i} w_i \left(\frac{\de \vol_{g_{tw_i}}}{\de \vol_{g_0}} - (1 - H_{\partial E_i} t w_i)\right)\de \vol_{g_0} \de t \right|
    \\
    &\overset{\eqref{eqn:FirstThing}}{\le} 
    \eps
    \int_0^1 \int_{\partial E_i}  \left|\frac{\de \vol_{g_{tw_i}}}{\de \vol_{g_0}} - (1 - H_{\partial E_i} t w_i)\right|\de \vol_{g_0} \de t
    \\
    &\overset{\eqref{eq:VolGzVolG0}}{\le}
    \eps\, \mathcal{C} \int_0^1 \int_{\partial E_i} t^2 w_i^2 \de \vol_{g_0} \de t\, ,
\end{split}
\end{equation}
and \eqref{eq:AlmostZeroMean} follows.
\medskip

{\textbf{Step 2.}} The goal of this second step is to prove that for every sufficiently large $i\in\N$ it holds
\begin{align}
    \label{eq:zzIntegraleLaplacianoDiffSimmetrica3}
    \frac12 \int_{\partial E_i}   |\nabla w_i|^2 - |\sff_{\partial E_i}|^2 w_i^2\de P_{\partial E_i} \le\,  & \eps \, \mathcal{C} \int_{\partial E_i} w_i^2 \de  P_{\partial E_i} \\
    &+ \int_{\partial E_i}\int_0^{w_i(x)}\int_0^z  \ric_{(s,x)}(\partial_z, \partial_z) \de s \de z \de \vol_{g_0} \, .\label{eq:badRicci}
\end{align}
The estimate will be obtained by applying the Gauss--Green theorem to the gradient of the signed distance function from $\partial E_i$ in the region bounded between $\partial E_i$ and $\partial E_i'$, carefully expanding all the terms and exploiting the assumption that $P(E_i)=P(E_i')$.

We start from the Gauss-Green identity
\begin{equation}\label{eq:zzIntegraleLaplacianoDiffSimmetrica}
    \int_{E_i'\Delta E_i} \Delta \dist_{\partial E_i} = \int_{\partial E_i'} g_i(\nabla \dist^s, - \nu_{E_i'} ) \de P_{\partial E_i'} - P(E_i)\, ,
\end{equation}
where $\nu_{E_i'}$ denotes the inner normal of $E_i'$ and $\dist^s$ is the signed distance from $E_i$.

Exploiting Fermi coordinates around $\partial E_i$ again, and recalling \eqref{eq:NormaleFermiCoordinates} and \eqref{eqn:FirstThing}, for $\eps$ small enough we can compute
\begin{equation*}
    \begin{split}
        \int_{\partial E_i'} &g_i(\nabla \dist^s, - \nu_{E_i'} ) \de P_{\partial E_i'}
        = \int_{\partial E_i'} \frac{1}{(1+ |\nabla^{w_i} w_i|_{w_i}^2(x))^{1/2}} \de P_{\partial E_i'}(w_i(x),x) \\
        \le &\int_{\partial E_i'} 1 - \frac12  |\nabla^{w_i} w_i|_{w_i}^2(x) + 2 |\nabla^{w_i} w_i|_{w_i}^4(x)  \de P_{\partial E_i'}(w_i(x),x) \\
        \overset{\eqref{eq:FormaVolumeFermiCoordinates}}{=}& P(E_i') \\
        &+ \int_{\partial E_i} \left(- \frac12  |\nabla^{w_i} w_i|_{w_i}^2(x) + 2 |\nabla^{w_i} w_i|_{w_i}^4(x)  \right) \left( 1+ |\nabla^{w_i} w_i|_{w_i}^2(x) \right)^{\frac12} \frac{\de \vol_{g_{w_i}}}{\de \vol g_0} \de \vol_{g_0}(x)\, .
    \end{split}
\end{equation*}
Expanding $|\cdot|_z$ with respect to $z$, it can be checked that $(1- \mathcal{C}|w_i|)  |\nabla w_i|^2 \le |\nabla^{w_i} w_i|_{w_i}^2 \le (1+ \mathcal{C}|w_i|)  |\nabla w_i|^2$, where $ |\nabla w_i|^2$ is understood to be computed with respect to $g_0$ (see the proof of \cref{lem:C1K}, and in particular the second of \eqref{eqn:VolumeAndGradientControl}). Hence, \eqref{eqn:FirstThing}, \eqref{eqn:ThirdThing}, and \eqref{eq:VolGzVolG0} imply
\begin{equation}
    \begin{split}
        \int_{\partial E_i'} g_i(\nabla \dist^s, - \nu_{E_i'} ) \de P_{\partial E_i'}
        &\le 
         P(E_i') - \frac12 \int_{\partial E_i}   |\nabla w_i|^2 \de P_{\partial E_i} + \eps \, \mathcal{C} \int_{\partial E_i} |\nabla w_i|^2 \de P_{\partial E_i} +\\
         &\qquad+
         \eps\,\mathcal{C} \int_{\partial E_i} |\nabla w_i|^4 \de P_{\partial E_i} \\
         &\overset{\eqref{eqn:ThirdThing}}{\le} 
         P(E_i') - \frac12 \int_{\partial E_i}   |\nabla w_i|^2 \de P_{\partial E_i} + \eps \, \mathcal{C} \int_{\partial E_i} w_i^2 \de  P_{\partial E_i}\, .
    \end{split}
\end{equation}
Since $P(E_i)=P(E_i')$, \eqref{eq:zzIntegraleLaplacianoDiffSimmetrica} is equivalent to
\begin{equation}\label{eq:zzIntegraleLaplacianoDiffSimmetrica2}
    \int_{E_i'\Delta E_i} \Delta \dist_{\partial E_i} \le - \frac12 \int_{\partial E_i}   |\nabla w_i|^2 \de P_{\partial E_i} + \eps \, \mathcal{C} \int_{\partial E_i} w_i^2 \de  P_{\partial E_i}\, .
\end{equation}

By the coarea formula we can also compute
\begin{equation}
    \begin{split}
        \int_{E_i'\Delta E_i} \Delta \dist_{\partial E_i}
        &= \int_{-\overline{\rho}}^{\overline{\rho}} \int_{\{ x \in \partial E_i \,:\, (z,x) \in E_i'\Delta E_i \}}\Delta \dist_{\partial E_i} (z,x) \de \vol_{g_z} \de z \\
        &= \int_{-\overline{\rho}}^{\overline{\rho}} \int_{\{w_i>0\}} \chi_{\{w_i>z\}}(x) \Delta \dist_{\partial E_i} (z,x) \frac{\de \vol_{g_z}}{\de \vol_{g_0}} \de \vol_{g_0}
         \de z
        + \\
        &\quad+
        \int_{-\overline{\rho}}^{\overline{\rho}} \int_{\{w_i<0\}} \chi_{\{w_i>z\}}(x) \Delta \dist_{\partial E_i} (z,x) \frac{\de \vol_{g_z}}{\de \vol_{g_0}} \de \vol_{g_0}
        \de z
        \\
        &=
        \int_{\{w_i>0\}}
        \int_{-\overline{\rho}}^{\overline{\rho}}
         \chi_{(0,w_i(x))}(z) \Delta \dist_{\partial E_i} (z,x) \frac{\de \vol_{g_z}}{\de \vol_{g_0}} \de z \de \vol_{g_0} +\\
         &\quad +
         \int_{\{w_i<0\}}
        \int_{-\overline{\rho}}^{\overline{\rho}}
         \chi_{(w_i(x),0)}(z) \Delta \dist_{\partial E_i} (z,x) \frac{\de \vol_{g_z}}{\de \vol_{g_0}} \de z \de \vol_{g_0} \\
         &= - \int_{\partial E_i} \int_0^{w_i(x)} 
         \Delta \dist^s (z,x) \frac{\de \vol_{g_z}}{\de \vol_{g_0}} \de z \de \vol_{g_0}\, .
    \end{split}
\end{equation}
Recall the evolution equation
\begin{equation}
\Delta \dist^s (z,x) = H_{\partial E_i}(x) + \int_0^z |\nabla^2 \dist^s|^2(s,x) + \ric_{(s,x)}(\partial_z, \partial_z) \de s\, .
\end{equation}
Since 
\begin{equation}
\left|  |\nabla^2 \dist^s|^2(s,x)- |\nabla^2 \dist^s|^2(0,x)\right| \le \mathcal{C}|s|\, 
\end{equation}
thanks to the scale invariant Riemann curvature bounds,
we thus find
\begin{equation*}
\begin{split}
    &\int_{E_i'\Delta E_i} \Delta \dist_{\partial E_i}
        =\\
        &=
        - \int_{\partial E_i} \int_0^{w_i(x)} 
         \left( H_{\partial E_i} + \int_0^z |\nabla^2 \dist^s|^2(s,x) + \ric_{(s,x)}(\partial_z, \partial_z) \de s\right) (1- H_{\partial E_i} z) \de z \de \vol_{g_0} +\\
         &\qquad
         - \int_{\partial E_i} \int_0^{w_i(x)} 
         \Delta \dist^s (z,x) \left(  \frac{\de \vol_{g_z}}{\de \vol_{g_0}} -(1- H_{\partial E_i}(x) z) \right)\de z \de \vol_{g_0}\\
         &\overset{\eqref{eq:VolGzVolG0}}{\ge}
         - \int_{\partial E_i} 
          H_{\partial E_i}w_i - H_{\partial E_i}^2 \frac{w_i^2}{2} +  |\sff_{\partial E_i}|^2 \frac{w_i^2}{2} \de \vol_{g_0} + \\
          &\qquad
          -\int_{\partial E_i}\int_0^{w_i(x)}(1- H_{\partial E_i} z) \int_0^z  \ric_{(s,x)}(\partial_z, \partial_z) \de s \de z \de \vol_{g_0} +\\
          &\qquad - \eps \, \mathcal{C} \int_{\partial E_i} w_i^2 \de \vol_{g_0}  \\
          &\overset{\eqref{eq:AlmostZeroMean}}{\ge}
          -\int_{\partial E_i}  \frac12 |\sff_{\partial E_i}|^2 w_i^2 \de \vol_{g_0}  -\int_{\partial E_i}\int_0^{w_i(x)}\int_0^z  \ric_{(s,x)}(\partial_z, \partial_z) \de s \de z \de \vol_{g_0} + \\
          &\qquad - \eps \, \mathcal{C} \int_{\partial E_i} w_i^2 \de \vol_{g_0} \, .
\end{split}
\end{equation*}
Recalling \eqref{eq:zzIntegraleLaplacianoDiffSimmetrica2} we get the claimed \eqref{eq:zzIntegraleLaplacianoDiffSimmetrica3}.
\medskip

\textbf{Step 3.} In this last step we are going to reach a contradiction to the strict volume preserving stability of $\partial E_i$ by estimating the Ricci curvature term in \eqref{eq:badRicci} with the help of \cref{lemma:maxRicci}.

We start by noting that
\begin{equation*}
    \begin{split}
    \int_{\partial E_i}&\int_0^{w_i(x)}\int_0^z  \ric_{(s,x)}(\partial_z, \partial_z) \de s \de z \de \vol_{g_0}
    \le \|w_i\|_\infty \int_{-\|w_i\|_\infty}^{\|w_i\|_\infty}  \int_{\partial E_i}\ric_{(s,x)}(\partial_z, \partial_z) \de s \de \vol_{g_0} \, .
    \end{split}
\end{equation*}
Rewriting the integral in \eqref{eq:stimamassimale} exploiting Fermi coordinates around $\partial E_i$ and using \eqref{eq:VolGzVolG0} as we did above, it is readily checked that \cref{lemma:maxRicci} implies
\begin{equation}
\int_{-\|w_i\|_\infty}^{\|w_i\|_\infty}  \int_{\partial E_i}\ric_{(s,x)}(\partial_z, \partial_z) \de s \de \vol_{g_0} \le \eps\, \mathcal{C} \|w_i\|_\infty + \mathcal{C}\|w_i\|_\infty^2\, .
\end{equation}
Hence
\begin{equation}\label{eq:zzIntegraleLaplacianoDiffSimmetrica4}
     \int_{\partial E_i}\int_0^{w_i(x)}\int_0^z  \ric_{(s,x)}(\partial_z, \partial_z) \de s \de z \de \vol_{g_0}
    \le \eps\, \mathcal{C} \|w_i\|_\infty^2 \overset{\eqref{eqn:SecondThing}}{\le} \eps\, \mathcal{C} \int_{\partial E_i} w_i^2 \de P_{\partial E_i}\, .
\end{equation}
Combining \eqref{eq:zzIntegraleLaplacianoDiffSimmetrica3} and \eqref{eq:zzIntegraleLaplacianoDiffSimmetrica4} we end up with
\begin{equation}\label{eq:zzIntegraleLaplacianoDiffSimmetrica5}
     \int_{\partial E_i}   |\nabla w_i|^2 - |\sff_{\partial E_i}|^2 w_i^2\de P_{\partial E_i} \le  \eps \, \mathcal{C} \int_{\partial E_i} w_i^2 \de  P_{\partial E_i}\, .
\end{equation}
Let $u_i:= w_i - \fint_{\partial E_i} w_i \de P_{\partial E_i}$. Notice that
\begin{equation*}
    \begin{split}
        \int_{\partial E_i} w_i^2 
        &= \int_{\partial E_i} \left( u_i^2 + \left(  \fint_{\partial E_i} w_i \right)^2 +2 u_i \fint_{\partial E_i} w_i \right)
        \le 
        2\int_{\partial E_i}  u_i^2  \de P_{\partial E_i}+ \frac32 \frac{1}{P(E_i)}\left(  \int_{\partial E_i} w_i \right)^2  \\
        & 
        \overset{\eqref{eq:AlmostZeroMean}}{\le} 2\int_{\partial E_i}  u_i^2  \de P_{\partial E_i}+ \eps \, \mathcal{C} \int_{\partial E_i} w_i^2  \de P_{\partial E_i}\, .
    \end{split}
\end{equation*}
Analogously,
\begin{equation*}
    \begin{split}
        \int_{\partial E_i}   |\sff_{\partial E_i}|^2 w_i^2
        &= \int_{\partial E_i}   |\sff_{\partial E_i}|^2 u_i^2 + 2 \fint_{\partial E_i} w_i \int_{\partial E_i} |\sff_{\partial E_i}|^2 u_i
        +  \left( \fint_{\partial E_i} w_i\right)^2 \int_{\partial E_i} |\sff_{\partial E_i}|^2 \\
        & 
        \overset{\eqref{eq:AlmostZeroMean}}{\le} \int_{\partial E_i}  |\sff_{\partial E_i}|^2 u_i^2  \de P_{\partial E_i}+ \eps \, \mathcal{C} \int_{\partial E_i} u_i^2  \de P_{\partial E_i}\, .
    \end{split}
\end{equation*}
Therefore \eqref{eq:zzIntegraleLaplacianoDiffSimmetrica5} can be rewritten as
\begin{equation}
    \int_{\partial E_i}   |\nabla u_i|^2 - |\sff_{\partial E_i}|^2 u_i^2\de P_{\partial E_i} \le  \eps \, \mathcal{C} \int_{\partial E_i} u_i^2 \de  P_{\partial E_i}\, .
\end{equation}
Observing that $\fint_{\partial E_i} u_i=0$, the strict stability stability estimate for the Jacobi operator from \cref{thm:strictstab} implies
\begin{equation}
\eps\, \mathcal{C} \int_{\partial E_i} u_i^2 \de P_{\partial E_i}
\ge  \eps_M\int_{\partial E_i} u_i^2  \de P_{\partial E_i}\, ,
\end{equation}
for some $\eps_M>0$ depending on $M$. Hence, if $\eps$ was taken sufficiently small, this implies that $w_i$ is constant. Since $w_i\not\equiv0$ as $E_i$ and $E_i'$ are distinct, this implies that $\vol_i(E_i)\neq \vol_i(E_i')$, where the volume is computed in the metric of $M_{V_i}$. This results into a contradiction, thus completing the proof of \cref{thm:unique}.

\subsection{Estimates for isoperimetric graphs}\label{sec:estimategraphs}

The goal of this section is to establish several technical estimates that will be useful later for the proof of \cref{prop:graphfunction}.
\medskip

We shall consider a smooth $n$-dimensional Riemannian manifold $(M,g)$ and we will assume that:
\begin{enumerate}
    \item\label{Condition1} There exist two isoperimetric sets $E_V, E'_V$ having equal volume $V>0$, having smooth boundary, and such that $H_{\partial E_V} = H_{\partial E_V'}$. Moreover we assume $\vartheta \leq H_{\partial E_V}\leq \eta$, for some $\vartheta,\eta>0$.

    \item\label{Condition2} $\Sigma':=\partial E_V'$ can be written as the normal graph over $\Sigma:=\partial E_V$ of a function $w:\Sigma\to\mathbb R$. Namely, 
    \begin{equation}
    \Sigma'=\{\exp_p(w(p)\nu_{\Sigma}(p)):p\in \Sigma\}\, ,
    \end{equation}
    where $\nu_\Sigma$ is the unit inner normal of $E_V$.

    \item\label{Condition3}  $\Sigma'\subset U$ where $U$ is an open neighborhood of $\Sigma$ where the metric $g$ can be written in Fermi coordinates with respect to $\Sigma$. Namely, there is $a>0$ such that
    \begin{equation}
    (U,g) \simeq ( (-a,a) \times \Sigma , \d z^2 + g_z ),
    \end{equation}
    where $g_z$ is the pull-back on $\Sigma$ via exponential map  of the induced metric on the parallel hypersurface
    \begin{equation}
    \Sigma_z:=\{\exp_p(z\nu_\Sigma(p)):p\in\Sigma\}\, ,
    \end{equation}
    for $z \in (-a,a)$. Let $\Psi_z:\Sigma\to \Sigma_z$ be the map $\Psi_z(p):=\exp_p(z\nu_\Sigma(p))$.

    \item\label{Condition4} There exists $K>1$ such that $ \|\sff_{\Sigma} \|_\infty+\sup_{p\in U}|\mathrm{Riem}(p)|\le K$ for every $p \in U$.

    \item\label{Condition5} There exists $\eps_0:=\eps_0(K,n,\vartheta,\eta,a)$ such that, if $C_1(K,a), C_2(n,K,\eta,a), C_3(n,K,a)$ denote the constants appearing in \cref{lem:C1K}, \eqref{eqnToApplyMoser}, and \eqref{eqn:Control3} below, respectively, then
    \begin{equation}\label{eqn:Control1}
    \frac 12 < (1-C_1\varepsilon_0)^{n/2} \leq (1+C_1\varepsilon_0)^{n/2} < 2\, ,
    \end{equation}
    \begin{equation}\label{eqn:Control2}
    \max\left\{\varepsilon_0,\varepsilon_0^2(1+C_1\varepsilon_0),C_1\varepsilon_0,C_2\varepsilon_0\right\}< \min\left\{\frac{1}{10},\frac{\vartheta}{C_3}\right\}\, ,
    \end{equation}
    and $\|w\|_\infty + \|\nabla w\|_\infty < \eps_0$. 
\end{enumerate}

We will denote by $g_w$ the metric on $\Sigma$ such that $(g_w)_p(v,v'):=g_{w(p)}(v,v')$, for every $v,v'\in T_p\Sigma$. We denote with $\nabla^w:=\nabla^{g_w}$ the gradient with respect to the metric $g_w$ on $\Sigma$, and with $|\cdot|_w$ the norm in the metric $g_w$. In this notation $g_0$ is the metric induced by $g$ on $\Sigma$. We stress that $\nabla $ denotes the gradient with respect to the metric $g_0$, and $|\cdot|$ is the norm with respect to the metric $g_0$. 

We will denote $\dot g_z:=\partial_zg_z$. Notice that $\dot g_z$ is a $(0,2)$-tensor on $\Sigma_z$, and $\dot g_z=2\sff_{\Sigma_z}$. We can identify $\dot g_z$ with a $(0,2)$-tensor on $\Sigma$ via the pull-back with the map $\Psi_z$ defined in \cref{Condition3}, namely, for every $p\in \Sigma$, and every $v,v'\in T_p\Sigma$, $(\dot g_z)_p(v,v')=2\sff_{\Sigma_z}(\mathrm{d}(\Psi_z)(v),\mathrm{d}(\Psi_z)(v'))$. We also denote $(\dot g_w)_p:=\dot g_{w(p)}$.

A standard computation exploiting Jacobi fields, \cref{Condition3} and \cref{Condition4} implies that there exists a constant $\mathcal{C}:=\mathcal{C}(K,a)$ such that for every $p\in \Sigma$, every $v\in T_p\Sigma$ with $(g_0)_p(v,v)=1$, and every $z\in (-a,a)$, it holds
\begin{equation}\label{eqn:ControlExponential}
g_z(v,v)\leq \mathcal{C}\, .
\end{equation}

We record the expression for the mean curvature of $\Sigma'$ in terms of $w$, borrowed from \cite{PacardSun}. It will be key for the subsequent estimates.

\begin{lemma}[{\cite[Proposition 4.1]{PacardSun}}]
Assume \cref{Condition2}, and \cref{Condition3} above hold. The mean curvature of $\Sigma'$ is 
\begin{equation}\label{eqn:MeanCurvatureGraph}
H_{\Sigma'}=-\mathrm{div}_{g_w}\left(\frac{\nabla^w w }{\left(1+|\nabla^w w|_w^2\right)^{\frac 12}}\right)+\frac 12 \left(1+|\nabla^w w|_w^2\right)^{\frac 12} \mathrm{tr}_{g_w}\dot g_w-\frac 12 \frac{\dot g_w(\nabla^w w,\nabla^w w)}{\left(1+|\nabla^w w|_w^2\right)^{\frac 12}}\, .
\end{equation}
\end{lemma}

In the next lemma we compare induced metrics and volume forms at different heights with respect to $\Sigma$ with the background metric $g_0$, with the help of the second fundamental form and Riemann curvature bounds.

\begin{lemma}\label{lem:C1K}
    Assume that \cref{Condition2}, \cref{Condition3}, and \cref{Condition4} above hold.
    There exists $C_1:=C_1(K,a)$ such that the following holds. For every $p\in \Sigma$, and every $v\in T_p\Sigma$ with $g_0(v,v)=1$,
    \begin{equation}\label{eqn:ContolOnGw}
    \begin{aligned}
        |(\dot g_w)_p(v,v)-(\dot g_0)_p(v,v)|&\leq C_1|w(p)|\, ,\\
        |(g_w)_p(v,v)-(g_0)_p(v,v)|&\leq C_1|w(p)|\, .
    \end{aligned}
    \end{equation}
    Moreover,
    \begin{equation}\label{eqn:VolumeAndGradientControl}
    \begin{aligned}
    (1-C_1\|w\|_{\infty})^{n/2} \mathrm{vol}_{g_0} \leq &\mathrm{vol}_{g_w}\leq (1+C_1\|w\|_{\infty})^{n/2}\mathrm{vol}_{g_0}\, ,\\
    |\nabla w|^2 (1- C_1|w|)\leq &|\nabla^w w|^2_w \leq |\nabla w|^2(1+C_1|w|)\, .
    \end{aligned}
    \end{equation}
\end{lemma}
\begin{proof}
    Let us fix some notation. Let $\Sigma^{n-1}$ be a hypersurface in $(M^n,g)$, let $\{e_1,\ldots,e_{n-1}\}$ be an orthonormal frame for $\Sigma$, and let $\mathrm{II}_\Sigma$ be the second fundamental form of $\Sigma$. We define $\sff_\Sigma^2:T\Sigma\times T\Sigma\to\mathbb R$ as follows: for every $X,Y\in T\Sigma$ we set $\sff_\Sigma^2(X,Y):=\sum_{k=1}^{n-1}g(\sff_\Sigma(X,e_k),\sff_\Sigma(Y,e_k))$.  
    
    Recall that the evolution equation of the second fundamental form $\sff_{\Sigma_z}$ is 
\begin{equation}\label{eqn:EvolutionEquation}\partial_z\sff_{\Sigma_z}+\sff_{\Sigma_z}^2+\mathrm{Riem}
    (\partial_z,\cdot,\partial_z,\cdot)=0\, .
    \end{equation}
    Hence, by using \cref{Condition4} and ODE comparison, there is $\mathcal{K}:=\mathcal{K}(K,a)$ such that $\|\sff_{\Sigma_z}\|_{\infty}\leq \mathcal{K}$ for every $z\in (-a,a)$.
    
    Notice that $\dot g_z = \partial_zg_z=2 \sff_{\Sigma_z}$. Via the pull back with the map $\Psi_z$ defined in \cref{Condition3} we are identifying $\dot g_z$ with a $(0,2)$-tensor on $\Sigma$ as explained above. 
    Hence, for every $z\in (-a,a)$, since $|\sect|$ and $\|\sff_{\Sigma_z}\|_{\infty}$ are uniformly bounded above by $K$, and since \eqref{eqn:ControlExponential}, and \eqref{eqn:EvolutionEquation} are in force, we have
    \begin{equation}\label{eqn:Control0}
    |\dot g_z-\dot g_0|\leq \int_0^{|z|} \left|\partial_z\dot g_z\right|=2\int_0^{|z|} \left|\partial_z\sff_{\Sigma_z}\right|\leq C_1(a,K)|z|\, .
    \end{equation}
    Here \eqref{eqn:Control0} has to be understood evaluated at every point $p\in\Sigma$, and at every $v\in T_p\Sigma$ with $(g_0)_p(v,v)=1$. Hence the first estimate of \eqref{eqn:ContolOnGw} holds. Similarly, we get the second inequality in \eqref{eqn:ContolOnGw} by using $\partial_zg_z=2\sff_{\Sigma_z}$, the fact that $\|\sff_{\Sigma_z}\|_{\infty}$ is uniformly bounded above by $\mathcal{K}$, and \eqref{eqn:ControlExponential}. Finally, the two inequalities in \eqref{eqn:VolumeAndGradientControl} are a direct consequence of the second estimate in \eqref{eqn:ContolOnGw}, possibly enlarging $C_1$.
\end{proof}

Our next goal is to exploit the constant mean curvature condition for $\Sigma'$ to obtain $L^{\infty}$--$L^{2}$ and gradient estimates for the graphing function $w$, via elliptic regularity theory. The precise statement follows.

\begin{lemma}\label{lemma:LinftyL2}
Assume \cref{Condition1}, \cref{Condition2}, \cref{Condition3}, \cref{Condition4}, and \cref{Condition5} above hold. Then there exists a constant $C:=C(n,K,a,V,\mathrm{diam}(\Sigma))>0$ such that the following hold. 
\begin{equation}\label{eqn:Moser}
\|w\|_{\infty} \leq C\|w\|_{L^2(\Sigma)}\, ,
\end{equation}
\begin{equation}\label{eqn:Caccioppoli}
\|\nabla w\|_{L^2(\Sigma)} \leq C\|w\|_{L^2(\Sigma)}\, .
\end{equation}
\end{lemma}

Before we can prove \cref{lemma:LinftyL2} we need one last technical estimate.

\begin{lemma}\label{lem:ToApplyMoser}
Assume that \cref{Condition1}, \cref{Condition2}, \cref{Condition3}, \cref{Condition4}, and \cref{Condition5} above hold. Then
    there exists a constant $C_2:=C_2(n,K,\eta,a)>0$ such that
    \begin{equation}\label{eqnToApplyMoser}
    \left|\mathrm{div}_{g_w}\left(\frac{\nabla^w w }{\left(1+|\nabla^w w|_w^2\right)^{\frac 12}}\right)\right|\leq C_2(|w|+|\nabla^w w|_w^2)\, .
    \end{equation}
\end{lemma}
\begin{proof}
    Let us first prove that 
    \begin{equation}\label{eqnToApplyMoser2}
    -\mathrm{div}_{g_w}\left(\frac{\nabla^w w }{\left(1+|\nabla^w w|_w^2\right)^{\frac 12}}\right)\leq C_2(|w|+|\nabla^w w|_w^2)\, .
    \end{equation}
    By \eqref{eqn:Control2} in \cref{Condition5}, and the second estimate in \eqref{eqn:VolumeAndGradientControl}, we have $|\nabla^w w|_w\leq 1$, and $|w|\leq 1$. Taking into account \eqref{eqn:ContolOnGw}, there is a constant $C_3:=C_3(n,K,a)>0$ such that 
    \begin{equation}\label{eqn:Control3}
    \tr_{g_w}\dot g_w\geq \tr_{g_0}\dot g_0 - 2C_3|w|\, .
    \end{equation}
    
    Moreover, taking into account that $H_\Sigma\geq \vartheta>0$, and that $C_3\varepsilon_0<\vartheta$ by \eqref{eqn:Control2}, we get $C_3\|w\|_{\infty}<H_\Sigma$. Hence, again recalling $|\nabla^w w|_w\leq 1$, we have
    \begin{equation}\label{eqn:EQ1}
    \begin{aligned}
    \frac 12\left(1+|\nabla^w w|_w^2\right)^{\frac 12}\mathrm{tr}_{g_w}\dot g_w &\stackrel{\eqref{eqn:Control3}}{\geq} \left(1+|\nabla^w w|_w^2\right)^{\frac 12}\left(\frac 12 \mathrm{tr}_{g_0}\dot g_0-C_3|w|\right)\\
    &=\left(1+|\nabla^w w|_w^2\right)^{\frac 12}\left(H_\Sigma-C_3|w|\right) \\
    &\geq \left(1+\frac 14 |\nabla^w w|^2_w\right)(H_\Sigma-C_3|w|)\geq H_\Sigma - \frac{5}{4}C_3 |w|\, .
    \end{aligned}
    \end{equation}
    In addition, using the second estimate in \eqref{eqn:ContolOnGw}, the fact that $\|\dot g_z\|_{\infty}=2\|\sff_{\Sigma_z}\|_{\infty}$ is uniformly bounded from above by $K$ by \cref{Condition4}, and \eqref{eqn:ControlExponential}, we get, for some $C_4:=C_4(K,a)>0$,
    \begin{equation}\label{eqn:EQ2}
    \frac 12 \frac{\dot g_w(\nabla^w w,\nabla^w w)}{\left(1+|\nabla^w w|_w^2\right)^{\frac 12}} \leq C_4|\nabla^w w|^2_w\, . 
    \end{equation}
    Using that $H_{\Sigma'}=H_\Sigma$, and combining \eqref{eqn:EQ1}, \eqref{eqn:EQ2}, and \eqref{eqn:MeanCurvatureGraph}, we get the sought \eqref{eqnToApplyMoser2} with $C_2:=\max\{C_4,2C_3\}$.
    Analogously, bounding from above $\mathrm{tr}_{g_w}\dot g_w$, and using $H\leq\eta$, one can also get 
    \[
\mathrm{div}_{g_w}\left(\frac{\nabla^w w }{\left(1+|\nabla^w w|_w^2\right)^{\frac 12}}\right)\leq C_2(n,K,\eta,a)(|w|+|\nabla^w w|_w^2)\, ,
    \]
    and thus the proof is concluded.    
\end{proof}

\begin{proof}[Proof of \cref{lemma:LinftyL2}]
    We claim that there exists $C(n,K,a)>0$ such that for every $p\geq 1$ it holds 
    \begin{equation}\label{eqn:Step0Moser}
        p\int_{\Sigma} w_+^{p-1}|\nabla w_+|^2\d \vol_{g_0} \leq C\int_{\Sigma} \left(w_+^{\frac{p+1}{2}}\right)^2 \d \vol_{g_0}\, ,
    \end{equation}
    where $w_+$ denotes the positive part of $w$.
    Recall that $|\nabla^w w|_w\leq 1$ since \eqref{eqn:Control2} and the second estimate in \eqref{eqn:VolumeAndGradientControl} hold. By exploiting the latter information, and integrating \eqref{eqnToApplyMoser} against $w_+^p$, we get
    \begin{equation}\label{eqn:Step1Moser}
        p\int_{\Sigma} w_+^{p-1}|\nabla^w w_+|_w^2\d \vol_{g_w} \leq \sqrt{2}C_2\int_{\Sigma}\left(w_+^{p+1}+w_+|\nabla^w w_+|_w^2w_+^{p-1}\right)\d \vol_{g_w}\, .
    \end{equation}
    By \eqref{eqn:Control2} we get $\sqrt{2}C_2w_+\leq \sqrt{2}C_2\varepsilon_0\leq p/2$, so we can absorb the term in the right hand side of \eqref{eqn:Step1Moser} to obtain
    \begin{equation}\label{eqn:Step1.5Moser}
        p\int_{\Sigma} w_+^{p-1}|\nabla^w w_+|_w^2\d \vol_{g_w} \leq 2\sqrt{2}C_2\int_{\Sigma}w_+^{p+1}\d \vol_{g_w}\, .
    \end{equation}
    By \eqref{eqn:Control1} and the first of \eqref{eqn:VolumeAndGradientControl}, we get that 
    \begin{equation}\label{eqn:Step2Moser}
    \frac 12 \vol_{g_w} \leq \vol_{g_0}\leq 2\vol_{g_w}\, .
    \end{equation}
    Moreover by \eqref{eqn:Control2} and the second estimate of \eqref{eqn:VolumeAndGradientControl}, we get
    \begin{equation}\label{eqn:Step3Moser}
        \frac 12 |\nabla w_+|^2 \leq |\nabla^w w_+|_w^2\, .
    \end{equation}
    By joining \eqref{eqn:Step2Moser} and \eqref{eqn:Step3Moser} with \eqref{eqn:Step1.5Moser}, we finally get the sought \eqref{eqn:Step0Moser}.
\medskip

    We now aim at obtaining \eqref{eqn:Moser} by exploiting \eqref{eqn:Step0Moser}, and a standard Moser iteration argument. By using the Gauss equations, we get $|\sect_{\Sigma}|\leq C(n,K)$, since $\|\sff_{\Sigma}\|_{\infty}+|\sect|\leq K$ by \cref{Condition4}. Hence there exists a constant $C(n,K,V,\mathrm{diam}(\Sigma))$ such that, calling $2^*:=(2n)/(n-2)$,
    \begin{equation}\label{eqn:SobolevPoincareInMoser}
    C\| f-\overline{f}\|_{L^{2^*}(\Sigma,\vol_{g_0})}\leq \|\nabla f\|_{L^2(\Sigma,\vol_{g_0})}, \qquad \forall f\in \mathrm{Lip}(\Sigma).
    \end{equation}
    From now on we abbreviate $\|\cdot\|_s:=\|\cdot\|_{L^s(\Sigma,\vol_{g_0})}$ for any $s\ge 1$. Let $q:=p+1\geq 2$. In the following computations the constants $C,\tilde C$, only depending on $n,K,a,V,\mathrm{diam}(\Sigma)$, might change from line to line. We have
    \begin{equation}\label{eqn:LongEquationMoser}
        \begin{aligned}
            \frac{q}{\sqrt{q-1}}\|w_+\|_q^{q/2} &\stackrel{\eqref{eqn:Step0Moser}}{\geq}C\|\nabla(w_+^{q/2})\|_2 \stackrel{\eqref{eqn:SobolevPoincareInMoser}}{\geq} C^2\left\|w_+^{q/2}-\overline{w_+^{q/2}}\right\|_{2^*}\geq C^2\left\|w_+^{q/2}\right\|_{2^*}-C^2\tilde C\left\|w_+^{q/2}\right\|_1 \\
            &=C^2\left\|w_+\right\|_{nq/(n-2)}^{q/2}-C^2\tilde C\left\|w_+^{q/2}\right\|_1 \geq C^2\left\|w_+\right\|_{nq/(n-2)}^{q/2}-C^2\tilde C^2\|w_+^{q/2}\|_2
            \\ & = C^2\left\|w_+\right\|_{nq/(n-2)}^{q/2} - C^2\tilde C\|w_+\|_q^{q/2}\, .
        \end{aligned}
    \end{equation}
   Set $q^*:=nq/(n-2)$. From \eqref{eqn:LongEquationMoser} we get 
    \begin{equation}\label{eqn:Step4Moser}
    \|w_+\|_{q^*} \leq \left(C\frac{q}{\sqrt{q-1}}+C\right)^{2/q}\|w_+\|_q\, .
    \end{equation}
    Recalling that $p\ge1$ was arbitrary, we can take $p=p_k$ such that $q=q_k=2\left(\frac{n}{n-2}\right)^k$ for $k\geq 0$. Hence applying \eqref{eqn:Step4Moser} with $q=q_k$, and letting $k\to\infty$, the usual Moser iteration scheme gives
    \begin{equation}\label{eqn:MoserFinal}
        \|w_+\|_{\infty}\leq C\|w_+\|_2\, .
    \end{equation}
    Reasoning analogously with $w_-$ in place of $w_+$, we get \eqref{eqn:MoserFinal} for $w_-$ as well. Hence \eqref{eqn:Moser} follows.
\medskip

    We now aim at proving \eqref{eqn:Caccioppoli}. Again, we can integrate \eqref{eqnToApplyMoser} against $w$. Arguing verbatim as we did to obtain \eqref{eqn:Step0Moser} (with $p=1$) we conclude 
    \begin{equation}
    \int_\Sigma |\nabla w|^2\d \vol_{g_0} \leq C \int_\Sigma w^2\d \vol_{g_0}\, ,
    \end{equation}
    which is the sought \eqref{eqn:Caccioppoli}.
\end{proof}

\subsection{Proof of \cref{prop:graphfunction}}
    Let $\mathcal{V}$ be the set of volumes $V\in (0,\infty)$ such that the isoperimetric profile $I_M$ is differentiable at $V$. Recalling that $I_M$ is concave \cite{AntonelliPasqualettoPozzettaSemola1}, then $(0,\infty)\setminus \mathcal{V}$ is a countable set. In particular, $\mathcal{V}$ has full measure in $(0,\infty)$. We divide the proof in two steps.
\medskip

    \textbf{Step 1}. Let $C$ be the constant in \eqref{eq:QCD}. We claim that there exist $\mathcal{C}:=\mathcal{C}(M)>0$, and $\varepsilon_0:=\varepsilon_0(C,\mathrm{AVR},n)>0$ such that the following holds. For every $\varepsilon<\varepsilon_0$, every sequence $\mathcal{V}\ni V_i\to\infty$, and every isoperimetric regions $E_{V_i},E_{V_i}'\subset M$ with $\vol(E_{V_i})=\vol(E_{V_i}')=V_i$, there is $i_0\in\mathbb N$ such that for every $i\geq i_0$, $\tilde\Sigma_{i}':=\partial E_{V_i}'$ is the normal graph over $\tilde\Sigma_i:=\partial E_{V_i}$ of a function $w_i:\tilde\Sigma_i\to\mathbb R$ satisfying
    \begin{equation}\label{eqn:FirstThingProof}
    \norm{w_i}_{\infty}\le \epsilon V_i^{\frac{1}{n}}\, ,\quad \norm{\nabla w_i}_{\infty}\le \epsilon \, ,
\end{equation}
    \begin{equation}\label{eqn:SecondThingProof}
        \norm{w_i}_{\infty}\le \mathcal{C}V_i^{-\frac{n-1}{2n}}\norm{w_i}_{L^2(\tilde\Sigma_i)}\, ,
    \end{equation}
    \begin{equation}\label{eqn:ThirdThingProof}
        \norm{\nabla w_i}_{L^2(\tilde\Sigma_i)}\le \mathcal{C}V_i^{-\frac 1n}\norm{w_i}_{L^2(\tilde\Sigma_i)}\, .
    \end{equation}
    
    Let $p\in M$ be a fixed point. Let $\varepsilon_0$ be small enough to be determined, and let $\varepsilon<\varepsilon_0$. Let us set $r_{V_i}:=\left(\rm{AVR}(M) \, \omega_n\right)^{-\frac{1}{n}}V_i^{\frac{1}{n}}$, and $M_i:=(M^n,r_{V_i}^{-1}\d_g,p)$. Let $\Sigma_i,\Sigma_i'\subset M_i$ denote the sets $\tilde\Sigma_i,\tilde\Sigma_i'\subset M$ in the rescaled manifold $M_i$. By the second estimate in \eqref{eq:Hausconv} we get that, for $i$ large enough,
    \begin{equation}\label{eqn:TheControlWeNeed}
    \begin{aligned}
    \d_{\rm{H}}^{M_i}(\Sigma_i,\Sigma_i')\leq \d_{\rm{H}}^{M_i}(\partial B_{1}^{M_i}(p),&\Sigma_i)+\d_{\rm{H}}^{M_i}(\partial B_{1}^{M_i}(p),\Sigma_i')\le \epsilon \,,\\
    B_{1}^{M_i}(\Sigma_i) &\cup B_1^{M_i}(\Sigma'_i) \subset B_3^{M_i}(p)\, .
    \end{aligned}
    \end{equation}
    By \eqref{eqn:ControlSecondForm}, \eqref{eq:QCD}, and the second line of \eqref{eqn:TheControlWeNeed}, we get that there is a constant $C_1:=C_1(C,\mathrm{AVR})$ only depending on the constant $C$ appearing in \eqref{eq:QCD} and $\mathrm{AVR}$ such that, for $i$ large enough,
    \begin{equation}
        \sup_{x\in \Sigma_i}|\sff(x)| + \sup_{x\in B_1^{M_i}(\Sigma_i)}|\mathrm{Riem}(x)|\leq C_1\, .
    \end{equation}
    Hence, if $\varepsilon_0$ is chosen small enough with respect to $C_1$, we can use the first line in \eqref{eqn:TheControlWeNeed}, and
    the results in \cref{appA} (in particular, we are using \eqref{eqn:GraphU}, \eqref{eqn:AllardGraph}, and \eqref{eqn:Convergence} as stated in the proof of \cref{prop:secondffstable} to $\Sigma_i,\Sigma_i'$) to get the existence of a function $\tilde w_i:\Sigma_i\to\mathbb R$ such that $\Sigma'_i$ is the normal graph over $\Sigma_i$ of $\tilde w_i$, and
    \begin{equation}\label{eqn:FuncGrad}
    \|\tilde w_i\|_{\infty,\Sigma_i}+\|\nabla^{\Sigma_i} \tilde w_i\|_{\infty,\Sigma_i}\leq \varepsilon\, .
    \end{equation}
    Also notice that, since $V_i\in\mathcal{V}$, it holds $H_{\Sigma_i}=H_{\Sigma'_i}\geq C_2$, where $C_2:=C_2(n,\mathrm{AVR})$ and the lower bound comes from \cite[Equation (3.18), Corollary 3.5]{AntonelliPasqualettoPozzettaSemola2}. Thanks to the bounds \eqref{eqn:GraphU}, \eqref{eqn:AllardGraph}, and \eqref{eqn:Convergence} as stated in the proof of \cref{prop:secondffstable}, \eqref{eqn:2SidedWell} holds. Hence \cref{prop:reachbound} applies. Thus \cref{Condition1}, \cref{Condition2}, \cref{Condition3}, and \cref{Condition4} in \cref{sec:estimategraphs} are satisfied, since we also have a uniform upper bound $H_{\Sigma_i}=H_{\Sigma'_i}\leq \tilde C(n,\mathrm{AVR})$ by \cite[Equation (3.17), Corollary 3.5]{AntonelliPasqualettoPozzettaSemola2}. Hence, choosing $\varepsilon_0:=\varepsilon_0(n,C,\mathrm{AVR})$ small enough, we get that \cref{Condition5} is also in force, due to \eqref{eqn:FuncGrad}. By \cref{lem:diaminj}, there is a constant $C_3$ depending on the manifold $M$ such that for $i$ large enough $\mathrm{diam}\Sigma_i\leq C_3$, where $\Sigma_i$ is endowed with the metric induced by the ambient $M_i$. Hence \cref{lemma:LinftyL2} applies and there exists a constant $\mathcal{C}:=\mathcal{C}(n,C,\mathrm{AVR},C_3)$ such that
\begin{equation}\label{eqn:SecondThingProof2}
        \norm{\tilde w_i}_{\infty,\Sigma_i}\le \mathcal{C}\norm{\tilde w_i}_{L^2(\Sigma_i)}\, ,
    \end{equation}
\begin{equation}\label{eqn:ThirdThingProof2}
        \norm{\nabla^{\Sigma_i} \tilde w_i}_{L^2(\Sigma_i)}\le \mathcal{C}\norm{\tilde w_i}_{L^2(\Sigma_i)}\, .
    \end{equation}
    Scaling back we get \eqref{eqn:FirstThingProof}, \eqref{eqn:SecondThingProof}, and \eqref{eqn:ThirdThingProof} are satisfied. This completes the proof of the claim.
    \medskip

    \textbf{Step 2}. Let $\mathcal{V}$ be as at the beginning of the proof, and let $\varepsilon_0,\mathcal{C}$ be as in the \textbf{Step 1} above. We claim that for every $\varepsilon<\varepsilon_0$ there exists $V_0:=V_0(\varepsilon)$ such that for every $\mathcal{V}\ni V > V_0$ the conclusion of \cref{prop:graphfunction} holds with the constant $\mathcal{C}$. This will be enough to conclude the proof.

    Suppose the latter is not true. Hence there is $\varepsilon<\varepsilon_0$ such that there is $\mathcal{V}\ni V_i\to \infty$ such that the conclusion of \cref{prop:graphfunction} with the constant $\mathcal{C}$ is false for every volume $V_i$. This contradicts the claim at the beginning of \textbf{Step 1}, thus completing the proof.

\subsection{Proof of \cref{lemma:maxRicci}}
For the proof of \cref{lemma:maxRicci} we will need to make effective some of the ideas that were introduced for the proof of \cref{lemma:Ricciiso0} above. More precisely, we are going to exploit the asymptotics of the isoperimetric profile of $M$ to estimate the Ricci curvature in the direction normal to the isoperimetric boundaries. With respect to the proof of \eqref{eq:estRiccistab} several additional error terms arise, that need to be carefully controlled.

\medskip

We will denote by $\mathcal{C}>0$ a constant depending on $M$ only that may change from line to line.
Consider a volume $V>0$ such that the isoperimetric profile $I$ of $M$ is differentiable at $V$. Let $E\subset M$ be any isoperimetric set of volume $V$. For $t \in \R$, we denote by $E_t$ the $t$-enlargement of $E$, namely $E_t:= \{ x \in M \st \dist^s_E <t \}$, where $\dist^s_E$ denotes the signed distance from $E$. By \cref{lem:PositiveReach} there exists $\overline{\rho}>0$ independent of $E, V$ such that the reach of $\partial E$ is bounded below by $\overline{\rho}V^{\frac1n}$, for any $V$ large enough. Hence, for $t \in (- \overline{\rho}V^{\frac1n}, \overline{\rho}V^{\frac1n})$ we can compute the perimeter $P(E_t)$ of $t$-enlargements as
\begin{align}
 \nonumber       P^{\frac{n}{n-1}}(E_t) 
        &= P^{\frac{n}{n-1}}(E) + \frac{n}{n-1}P^{\frac{1}{n-1}}(E) t \int_{\partial E} H_{\partial E} + \int_0^t \int_0^s \frac{\d^2}{\d r^2} P^{\frac{n}{n-1}}(E_r) \de r \de s \\
\nonumber        &= P^{\frac{n}{n-1}}(E) + \frac{n\, P^{\frac{n}{n-1}}(E) \, H_{\partial E}}{n-1}\, t + 
        \frac{n}{n-1} \int_0^t \int_0^s 
        \frac{ P^{\frac{2-n}{n-1}}(E_r)}{n-1} \left( \int_{\partial E_r} H_{\partial E_r} \right)^2\de r \de s  +\\
        &\qquad+ 
        \frac{n}{n-1} \int_0^t \int_0^s
         P^{\frac{1}{n-1}}(E_r) \int_{\partial E_r} H_{\partial E_r}^2 - |\sff_{\partial E_r}|^2 - \ric(\nabla \dist^s_E, \nabla \dist^s_E) 
        \de r \de s \, .\label{eq:firstexpansion per}
\end{align}
Note that the last term appearing in \eqref{eq:firstexpansion per} is exactly the term we would like to estimate.
\medskip

Let $J(\cdot):=I^{\frac{n}{n-1}}(\cdot)$, where $I$ is the isoperimetric profile of $M$. Since $E$ is isoperimetric, denoting for brevity by $|\cdot|:= \vol$ the Riemannian volume on $M$, for $t \in (0, \overline{\rho}V^{\frac1n})$ we can estimate
\begin{equation}\label{eq:EspansioneSecondOrderJ}
    \begin{split}
        J&(|E_t|) + J(|E_{-t}|) - 2 J(V)
        \le P^{\frac{n}{n-1}}(E_t) + P^{\frac{n}{n-1}}(E_{-t}) -2 P^{\frac{n}{n-1}}(E)
        \\
        &=  
        \frac{n}{n-1} \int_0^t \int_{-s}^s 
        \frac{ P^{\frac{2-n}{n-1}}(E_r)}{n-1} \left( \int_{\partial E_r} H_{\partial E_r} \right)^2 \\
        &+
         P^{\frac{1}{n-1}}(E_r) \int_{\partial E_r} H_{\partial E_r}^2 - |\sff_{\partial E_r}|^2 - \ric(\nabla \dist^s_E, \nabla \dist^s_E) 
        \de r \de s \\
    &\le \frac{n}{n-1} \int_0^t \int_{-s}^s
         P^{\frac{1}{n-1}}(E_r) \left( \int_{\partial E_r} H_{\partial E_r}^2 \right) - P^{\frac{1}{n-1}}(E_r) \int_{\partial E_r} \ric(\nabla \dist^s_E, \nabla \dist^s_E) 
        \de r \de s\,  .
    \end{split}
\end{equation}
By the isoperimetric inequality on manifolds with nonnegative Ricci and Euclidean volume growth, see e.g. \cite[Corollary 1.3]{BrendleIsoperimetricIneq}, 
and taking into account \cite[Corollary 3.13, Corollary 4.19]{AntonelliPasqualettoPozzettaSemola1} and \cite[Corollary 3.6]{AntonelliBrueFogagnoloPozzetta}, for any $V$ large enough we can further estimate
\begin{equation}\label{eq:zzz}
\begin{split}
    C\,P^{\frac{n}{n-1}}(E_r) &\ge  |E_r| \ge |E_{-|r|}|
\ge V - P(E) \int_0^{|r|} \left( 1- \frac{H_{\partial E}}{n-1} z \right)_+^{n-1} \de z \ge V - c V^{\frac{n-1}{n}}\overline{\rho}V^{\frac1n}
\end{split}
\end{equation}
for some $C, c$ depending on $n, {\rm AVR}(M)$, for any $r \in(- \overline{\rho}V^{\frac1n}, \overline{\rho}V^{\frac1n})$. Hence, up to possibly decreasing $\overline{\rho}$, \eqref{eq:zzz} implies that $P^{\frac{1}{n-1}}(E_r) \ge \mathcal{C} V^{\frac1n}$, for any $r \in(- \overline{\rho}V^{\frac1n}, \overline{\rho}V^{\frac1n})$ and any $V$ large enough. Observe also that the previous argument implies that
\begin{equation}\label{eq:RangeVolumiEr}
    \frac34 V < |E_r| < \frac54 V\, ,
\end{equation}
for any $r \in(- \overline{\rho}V^{\frac1n}, \overline{\rho}V^{\frac1n})$ and any $V$ large enough, up to decreasing $\overline{\rho}$.

Moreover, computing integrals in Fermi coordinates as in the first part of the proof of \cref{thm:unique} (see \eqref{eq:FormaVolumeFermiCoordinates}), it follows from \cref{prop:graphfunction} that
\begin{equation}\label{eq:zzza}
    \begin{split}
    \int_{\partial E_r} \ric(\nabla \dist^s_E, \nabla \dist^s_E)
        \ge \mathcal{C} \int_{\partial E} {\rm ric}(r,x) \de P_{\partial E}(x)\, ,
    \end{split}
\end{equation}
for any $r \in(- \overline{\rho}V^{\frac1n}, \overline{\rho}V^{\frac1n})$, where ${\rm ric}(r,x) := \ric (\nabla \dist^s_E, \nabla \dist^s_E) \circ (\exp_x(r \nu^{\rm ext}_E(x)))$ for any $x \in \partial E$, where $\nu^{\rm ext}_E$ is the unit outer normal along $\partial E$. 

Combining \eqref{eq:zzz} and \eqref{eq:zzza}, for $t \in (0, \overline{\rho}V^{\frac1n})$ we get
\begin{equation*}
    \begin{split}
        \int_0^t \int_{-s}^s &
          P^{\frac{1}{n-1}}(E_r) \int_{\partial E_r} \ric(\nabla \dist^s_E, \nabla \dist^s_E) 
        \de r \de s 
        \ge \mathcal{C}V^{\frac1n} \int_0^t \int_{-s}^s
          \int_{\partial E} {\rm ric} (r,x) \de P_{\partial E}(x) 
        \de r \de s  \\
        &\ge \mathcal{C}V^{\frac1n} \int_{-t}^t \int_{|r|}^t
          \int_{\partial E} {\rm ric} (r,x) \de P_{\partial E}(x) 
        \de s \de r 
        = \mathcal{C}V^{\frac1n} \int_{-t}^t (t-|r|)
          \int_{\partial E} {\rm ric} (r,x) \de P_{\partial E}(x) 
         \de r \\
         &\ge \mathcal{C}V^{\frac1n} \int_{-\frac{t}{2}}^{\frac{t}{2}} (t-|r|)
          \int_{\partial E} {\rm ric} (r,x) \de P_{\partial E}(x) 
         \de r \ge \mathcal{C}V^{\frac1n} \, \frac{t}{2} \int_{-\frac{t}{2}}^{\frac{t}{2}} 
          \int_{\partial E} {\rm ric} (r,x) \de P_{\partial E}(x) 
         \de r \, .
    \end{split}
\end{equation*}
Combining with \eqref{eq:EspansioneSecondOrderJ} we find
\begin{equation}\label{eq:EspansioneSecondOrderJ2}
    \begin{split}
        t\, V^{\frac1n} \int_{-\frac{t}{2}}^{\frac{t}{2}} 
          \int_{\partial E} {\rm ric} (r,x) \de P_{\partial E}(x) 
         \de r  
         &\le
         \mathcal{C} \bigg( \frac{n}{n-1} \int_0^t \int_{-s}^s
         P^{\frac{1}{n-1}}(E_r) \left( \int_{\partial E_r} H_{\partial E_r}^2 \right)  \de r \de s +
         \\
         &\qquad\qquad-(J(|E_t|) + J(|E_{-t}|) - 2 J(V))\bigg)\, .
    \end{split}
\end{equation}

We now aim at estimating the difference $J(|E_t|) + J(|E_{-t}|) - 2 J(V)$. This is the goal of the following lemma, whose proof is postponed to the end of this section.

\begin{lemma}\label{lemma:stimaJ}
In the same assumptions of \cref{lemma:maxRicci} and notation as above, for any $\eta \in (0,1)$ and $\eps>0$ there exists $\mathcal{G}_\eps\subset (0,\infty)$ such that \eqref{eq:DensityOneVolumiRicciBasso} holds and
\begin{align}\label{eq:EspansioneSecondOrderJ3}
  \nonumber   -(J(|E_t|) + J(|E_{-t}|) - 2 J(V))\le & \,  \mathcal{C}\eta P(E) \eps t^2 V^{-\frac1n} \\
     &-
        \int_0^t \int_{-s}^s \frac{n}{n-1} I^{\frac{1}{n-1}}(|E_r|) I'_+(|E_r|)  \left( \int_{\partial E_r} H_{\partial E_r} \right)\de r \de s\, 
\end{align}
for any $V \in \mathcal{G}_\eps$\footnote{We only stress dependence with respect to $\eps$ because $\eta$ is eventually chosen to be a constant depending on $M$.}.
\end{lemma}

Let $\eps$ be as in the assumptions of \cref{lemma:maxRicci}. We fix $\eta=\eta(M) \in (0,1)$, whose choice will be clear at the end of the proof of \cref{lemma:maxRicci} (more precisely, see \eqref{eq:lastproofmax} below). Hence let $\mathcal{G}_\eps$ be given by \cref{lemma:stimaJ}. We continue the proof of \cref{lemma:maxRicci} assuming that $V \in \mathcal{G}_\eps$.

Combining \eqref{eq:EspansioneSecondOrderJ2} and \eqref{eq:EspansioneSecondOrderJ3} we deduce
\begin{align}
\nonumber         t\, V^{\frac1n} &\int_{-\frac{t}{2}}^{\frac{t}{2}} 
          \int_{\partial E} {\rm ric} (r,x) \de P_{\partial E}(x) 
         \de r  
         \le \mathcal{C} \bigg(
         \mathcal{C}\eta P(E) \eps t^2 V^{-\frac1n}+ \\
         &+
         \frac{n}{n-1} \int_0^t \int_{-s}^s
         P^{\frac{1}{n-1}}(E_r) \left( \int_{\partial E_r} H_{\partial E_r}^2 \right) 
        - I^{\frac{1}{n-1}}(|E_r|) I'(|E_r|)  \left( \int_{\partial E_r} H_{\partial E_r} \right)  \de r \de s
         \bigg)\, , \label{eq:EspansioneSecondOrderJ4}
\end{align}
for $t \in (0,\overline{\rho}V^{\frac1n})$.
\smallskip

The proof of \cref{lemma:maxRicci} will be completed with the help of the following: 

\begin{lemma}\label{lemma:convuniffinal}
In the same assumptions of \cref{lemma:maxRicci} and notation as above, for any $\eps,\eta>0$ there exists $\hat{V}=\hat{V}(\eps,\eta,M)>0$ such that
\begin{equation}\label{eq:zzConvergenzaUniformeFinale}
    \sup_{r \in (-\overline{\rho}V^{\frac1n}/2,\overline{\rho}V^{\frac1n}/2)}
    P^{\frac{1}{n-1}}(E_r) \left( \int_{\partial E_r} H_{\partial E_r}^2 \right) 
        - I^{\frac{1}{n-1}}(|E_r|) I'_+(|E_r|)  \left( \int_{\partial E_r} H_{\partial E_r} \right) \le \eps \, \eta \, V^{\frac{n-2}{n}}\, ,
\end{equation}
for any $V \ge \hat{V}$.
\end{lemma}

Applying \cref{lemma:convuniffinal}, whose proof is postponed to the end of this section, with $\eps$ as in the assumptions of \cref{lemma:maxRicci} and $\eta$ as fixed above, we can combine \eqref{eq:EspansioneSecondOrderJ4} and \eqref{eq:zzConvergenzaUniformeFinale} and obtain the following. For $t \in (0,\overline{\rho}V^{\frac1n}/2)$, for any $V \in \mathcal{G}_\eps$, up to replacing $\mathcal{G}_\eps$ with $\mathcal{G}_\eps \cap (\hat{V},\infty)$, it holds
\begin{equation}\label{eq:EspansioneSecondOrderJ5}
    \begin{split}
        t\, V^{\frac1n} \int_{-\frac{t}{2}}^{\frac{t}{2}} 
          \int_{\partial E} {\rm ric} (r,x) \de P_{\partial E}(x) 
         \de r  
         &\le \mathcal{C} \left(
         \eta P(E) \eps t^2 V^{-\frac1n}+ 
            \eps \eta V^{\frac{n-2}{n}} t^2 \right) \\
            &\le 
        \mathcal{C} 
         \eta P(E) \eps t^2 V^{-\frac1n}\,  .
    \end{split}
\end{equation}
Exploiting Fermi coordinates again, arguing similarly as in \eqref{eq:zzza}, we conclude that
\begin{equation}\label{eq:lastproofmax}
V^{\frac2n} \fint_{B_r(\partial E)} \ric(\nabla \dist^s_E, \nabla \dist^s_E) \le \mathcal{C}\eta \eps \le \eps\, ,
\end{equation}
for any $r \in (0, C_0 V^{\frac1n})$ for some $C_0>0$. Note that in \eqref{eq:lastproofmax} above we estimated $\mathcal{C}\eta \le1$ for a choice of $\eta$ depending only on the last constant $\mathcal{C}$ appeared, hence depending only on $M$. This completes the proof of \cref{lemma:maxRicci}, modulo the proofs of \cref{lemma:stimaJ} and \cref{lemma:convuniffinal} whose details are discussed below.
\medskip

\begin{proof}[Proof of \cref{lemma:stimaJ}]
Denote by $D^2 J$ the (negative) finite measure given by the distributional second derivative of $J$.   For $W>4$ we consider the maximal function $M_{W}$ of $\chi_{(W/4,+\infty)} |D^2J| $, namely
\begin{equation}
M_{W}(v) := \sup_{h>0} \frac{1}{2h}|D^2J|((v-h,v+h)\cap (W/4,\infty)).
\end{equation}
Denoting by $|A|:=\mathcal{L}^1(A)$ for any measurable set $A \subset \R$, the classical Hardy--Littlewood inequality (whose proof works also for measures, cf. \cite[Theorem 3.5.6]{HKST})  implies
\begin{equation}\label{eq:zzMaximal}
\begin{split}
    &\left| \left\{ 
v \in \R \st M_{W}(v) > W^{-1} \left( |D^2J|(W/4 , +\infty) \right)^{\frac12}
\right\} \right|
\frac{ \left(|D^2J|(W/4 , +\infty) \right)^{\frac12}}{W}\\
&\qquad:=|X_W|\frac{ \left(|D^2J|(W/4 , +\infty) \right)^{\frac12}}{W} \le 5 |D^2J|(W/4, +\infty)\, .
\end{split}
\end{equation} 
Notice that if $v\notin X_W$, then 
\begin{equation}\label{eq:Massimale1}
    \sup_{h>0} \frac{1}{2h}|D^2 J| ((v-h,v+h) \cap (W/4,\infty))
    \le W^{-1} \left( |D^2J|(W/4 , +\infty) \right)^{\frac12}\, .
\end{equation}
Define recursively $W_1=5$, $W_{i+1}=2W_i$, and let $\widetilde X_i = X_{W_i} \cap (W_i, W_{i+1})$.
\smallskip

Let $\eps>0$, $\eta \in (0,1)$.  There exists $i_\eps\ge 1$ such that $\left( |D^2J|(W_i/4 , +\infty) \right)^{\frac12} \le \eps\,\eta$ for any $i \ge i_\eps$\footnote{We only stress dependence with respect to $\eps$ because $\eta$ is eventually chosen to be a constant depending on $M$.}. We claim that the set $\mathcal G_\eps$ can be chosen as
\begin{equation}
    \mathcal G_\eps
    := \{ V>W_{i_\eps} \st \exists\, I'(V) \}  \setminus \bigcup_{i \ge i_\eps} \widetilde X_i\, .
\end{equation}

Indeed, we first note that for any $z> W_{i_\eps}$ large, writing $W_{i_z}\le z < W_{i_z+1}$ for some $i_z$, it holds
\begin{equation*}
\begin{split}
    \left| (z,2z) \cap\bigcup_{i \ge i_\eps} \widetilde X_i\right| &\le \left| \widetilde X_{i_z} \right| +  \left| \widetilde X_{i_z+1} \right| 
    \\&\le 5 W_{i_z}\left( |D^2J|(W_{i_z}/4 , +\infty) \right)^{\frac12} +
    5 W_{i_z+1}\left( |D^2J|(W_{i_z+1}/4 , +\infty) \right)^{\frac12} \\
    &\le 10 W_{i_z+1} \left( |D^2J|(W_{i_z}/4 , +\infty) \right)^{\frac12} \\
    &\le     20 z  \left( |D^2J|(z/8 , +\infty) \right)^{\frac12}\, .
\end{split}
\end{equation*}
Therefore $\lim_{z\to+\infty} \mathcal{L}^1\left( (z,2z) \cap (\cup_{i \ge i_\eps} \widetilde X_i)\right)/z =0$ and \eqref{eq:DensityOneVolumiRicciBasso} follows.
\smallskip

Hence let $V \in \mathcal{G}_\eps$ and let $E \subset M$ be any isoperimetric set of volume $V$, as considered above. Let $i_0$ be such that $W_{i_0} < V \le W_{i_0+1}$. Recalling \eqref{eq:RangeVolumiEr}, let $h_r:= \max\{V-|E_{-r}|, |E_r|-V\}< V/4$ for $r \in (0,\overline{\rho} V^{\frac1n})$. 
Since $V-h_r \ge \tfrac34 V > W_{i_0}/2 $, \eqref{eq:Massimale1} implies
\begin{equation*}
    \begin{split}
        |D^2J|(|E_{-r}|, |E_r|)
        &\le |D^2J|(V-h_r,V+h_r) = |D^2J|((V-h_r,V+h_r) \cap (W_{i_0}/4, \infty) ) \\
        &\overset{\eqref{eq:Massimale1}}{\le}
        \frac{2 h_r}{W_{i_0}} \eps\,\eta 
        \le \frac{4 h_r}{V} \eps\,\eta\, .
    \end{split}
\end{equation*}
Moreover, arguing similarly as in \eqref{eq:zzz}, one estimates $h_r \le \mathcal{C} P(E) r$ (see \cite[Corollary 3.13]{AntonelliPasqualettoPozzettaSemola1}). Hence
\begin{equation}\label{eq:Massimale2}
     |D^2J|(|E_{-r}|, |E_r|) \le \mathcal{C} \eta P(E) \frac{\eps}{V} r\, ,
\end{equation}
for $r \in (0,\overline{\rho} V^{\frac1n})$.

Let us now conclude the proof with a mollification argument. Let $\rho_\tau:\R\to\R$, for $\tau\in(0,1)$ be a family of non-negative symmetric mollifiers with support in $(-\tau,\tau)$.
Let $J_\tau:= J * \rho_\tau:(1,\infty) \to \R$.  Hence $J_\tau$ is smooth, and $J_\tau'' = D^2 J * \rho_\tau$. Since $J$ is concave, $D^2 J\leq 0$ in the sense of distributions. Since the convolution with non-negative mollifiers preserves the sign of a distribution we also get that $J_\tau''\leq 0$. Thus $J_\tau$ is concave.

Performing a Taylor expansion of $J_\tau(|E_t|)$ with respect to $t$  as done in \eqref{eq:firstexpansion per} for $P^{\frac{n}{n-1}}(E_t)$, using 
\cite[Proposition 3.11]{AntonelliPasqualettoPozzettaSemola1} and \cite[Corollary 3.6]{AntonelliBrueFogagnoloPozzetta}, up to taking a larger $i_\eps$, for $t \in (0,\overline{\rho}V^{\frac1n})$ we can estimate
\begin{equation*}
    \begin{split}
        -(J_\tau&(|E_t|) + J_\tau(|E_{-t}|) - 2 J_\tau(V)) = \int_0^t \int_{-s}^s |J_\tau''|(|E_r|) P(E_r)^2 - J_\tau'(|E_r|) \left( \int_{\partial E_r} H_{\partial E_r} \right)\de r \de s \\
        &\le \left( 1+ \frac{H_{\partial E}}{n-1}\overline{\rho} V^{\frac{1}{n}}\right)P(E)\int_0^t \int_{-|E_s|}^{|E_s|} |J_\tau''|(v) \de v \de s 
        -\int_0^t \int_{-s}^s J_\tau'(|E_r|) \left( \int_{\partial E_r} H_{\partial E_r} \right)\de r \de s \\
        & \le \mathcal{C}P(E) \int_0^t \int_{-|E_s|}^{|E_s|} |J_\tau''|(v) \de v \de s 
        -\int_0^t \int_{-s}^s J_\tau'(|E_r|) \left( \int_{\partial E_r} H_{\partial E_r} \right)\de r \de s \, .
    \end{split}
\end{equation*}
Letting $\tau\to0$ we get
\begin{align}
 \nonumber       -(J(|E_t|) + J(|E_{-t}|) - 2 J(V))
        \overset{\eqref{eq:Massimale2}}{\le}&
        \mathcal{C}\eta P(E)^2 \frac{\eps}{V} t^2 - \int_0^t \int_{-s}^s J'_+(|E_r|) \left( \int_{\partial E_r} H_{\partial E_r} \right)\de r \de s \\
  \nonumber      \le &\, \, \mathcal{C}\eta P(E) \eps t^2 V^{-\frac1n}\\
      &  -
        \int_0^t \int_{-s}^s \frac{n}{n-1} I^{\frac{1}{n-1}}(|E_r|) I'_+(|E_r|)  \left( \int_{\partial E_r} H_{\partial E_r} \right)\de r \de s\, .
\end{align}   
\end{proof}

\medskip

\begin{proof}[Proof of \cref{lemma:convuniffinal}]

Note that \eqref{eq:zzConvergenzaUniformeFinale} is a scale invariant estimate. Hence to prove it we can argue by contradiction and assume that there exist sequences $V_i\nearrow\infty$, $r_i \in  (-\overline{\rho}V_i^{\frac1n}/2,\overline{\rho}V_i^{\frac1n}/2)$, and isoperimetric sets $E_i$ of volume $V_i$ such that
\begin{equation}
 P^{\frac{1}{n-1}}(E_{r_i}) \left( \int_{\partial E_{r_i}} H_{\partial E_{r_i}}^2 \right) 
        - I^{\frac{1}{n-1}}(|E_{r_i}|) I'_+(|E_{r_i}|)  \left( \int_{\partial E_{r_i}} H_{\partial E_{r_i}} \right) > \eps\, \eta V_i^{\frac{n-2}{n}}\, ,
\end{equation}
where $E_{r_i}$ is the $r_i$-enlargement of $E_i$ in $M$.
In the sequence of rescaled manifolds $M_{V_i}$ this means that there exists a sequence of isoperimetric sets $\widetilde E_i\subset M_{V_i}$ such that $|\widetilde E_i| = {\rm AVR}(M) \omega_n$, and a sequence $\widetilde r_i \in (-\widetilde{\rho},\widetilde{\rho})$ such that
\begin{equation}\label{eq:zzFinale0}
 P_i^{\frac{1}{n-1}}(\widetilde E_{\widetilde r_i}) \left( \int_{\partial \widetilde E_{\widetilde r_i}} H_{\partial \widetilde E_{\widetilde r_i}}^2 \right) 
        - I_i^{\frac{1}{n-1}}(| \widetilde E_{\widetilde r_i}|) (I_i)_+'(|\widetilde E_{\widetilde r_i}|)  \left( \int_{\partial \widetilde E_{\widetilde r_i}} H_{\partial\widetilde E_{\widetilde r_i}} \right) > \mathcal{C}\eps\, \eta \, ,
\end{equation}
where perimeters, volumes, distances and isoperimetric profile are computed on $M_{V_i}$. 
\smallskip

The key for the proof will be the stability of the signed distance functions from isoperimetric sets and their Laplacians under Gromov--Hausdorff convergence.

\smallskip

Up to the extraction of a subsequence that we do not relabel, $M_{V_i}$ converges to an asymptotic cone $X$, with isoperimetric profile $I_X$, and $\widetilde r_i\to r_\infty$. Moreover $I_i\to I_X$ locally uniformly on $(0,\infty)$. Since the right derivative $(I_i^{\frac{n}{n-1}})_+'$ converges to the constant $n^{\frac{n}{n-1}}(\theta \omega_n)^{\frac{1}{n-1}}$ almost everywhere and $I_i^{\frac{n}{n-1}}$ is concave, where $\theta:={\rm AVR}(M)$, then also $(I_i^{\frac{n}{n-1}})_+'\to n^{\frac{n}{n-1}}(\theta \omega_n)^{\frac{1}{n-1}}$ locally uniformly on $(0,\infty)$. Hence $I_i^{\frac{1}{n-1}}(I_i)'_+\to (n\theta \omega_n)^{\frac{1}{n-1}}(n-1)$ locally uniformly on $(0,\infty)$. 
\smallskip

Denoting $H_i:= H_{\partial \widetilde E_i}$, we know from \cite[Theorem 4.23]{AntonelliPasqualettoPozzettaSemola1} that $\widetilde E_{\widetilde r_i}$ converges to the ball $B_{(1+r_\infty)}^X$ in $X$ centered at the tip, and $H_i\to n-1$. Moreover by \cite[Proposition 3.11]{AntonelliPasqualettoPozzettaSemola1} we get
\begin{equation}\label{eq:zzFinale1}
\begin{split}
    n\theta\omega_n (1+r_\infty)^{n-1} &= P_X(B^X_{1+r_\infty}) \le \liminf_i P_i( \widetilde E_{\widetilde r_i})\\
    &\le \liminf_i \left( 1 + \frac{H_i}{n-1}r_\infty\right)^{n-1} P_i (\widetilde E_i) \\
    &=  n\theta\omega_n (1+r_\infty)^{n-1} \, .
\end{split}
\end{equation}
Finally, by \cite[Theorem 1.2]{AntonelliPasqualettoPozzettaSemola1},
\begin{equation}\label{eq:zzLaplacianBounds}
    \Delta f_i - \frac{H_i}{1+ \frac{H_i}{n-1} f_i} \ge 0 \quad\text{on $\{0\ge f_i >-2\widetilde\rho\}$}\, ,
    \quad
    \frac{H_i}{1+ \frac{H_i}{n-1} f_i} -  \Delta f_i \ge 0 \quad\text{on $\{0\le f_i <2\widetilde\rho\}$}\, ,
\end{equation}
where $f_i$ is the signed distance function from $\widetilde E_i$. Moreover, the Riccati equation for the evolution of $\Delta f_i$ along geodesics perpendicular to $\partial \widetilde E_i$, together with \cref{prop:secondffstable}, implies that $\Delta f_i$ is uniformly bounded on $\{-2\widetilde \rho < f_i < 2\widetilde \rho\}$  (see also the proof of \cref{lem:C1K}). By \cite[Proposition 1.3.1]{AmbrosioHondastab}, it follows that, up to subsequences, $\Delta f_i$ converges weakly in $L^2$ in the sense of \cite{AmbrosioHondastab} along the sequence of spaces $(X_{i,\widehat\rho}:=\{-\widehat \rho \le f_i \le \widehat \rho\}, \dist_{M_{V_i}}) \to (\overline{B}^X_{1+\widehat \rho} \setminus B^X_{1- \widehat \rho} , \dist_X)$
for any $\widehat \rho \in (0,2\widetilde\rho)$. By the divergence theorem, it follows that the $L^2$-weak limit of $\Delta f_i$ coincides with the Laplacian from the tip $o$ on $X$, i.e., the function $\frac{n-1}{\dist^X_o}$.

Letting $\chi_i = {\rm sgn}(f_i)$, then $\chi_i \to \chi$ in $L^2$-strong where $\chi=1$ on $X \setminus B^X_1$ and $\chi=-1$ on $B^X_1$. It follows from \eqref{eq:zzLaplacianBounds} that
\begin{equation}
    \begin{split}
        \int_{X_{i,\widehat\rho}} \bigg|  \frac{H_i}{1+ \frac{H_i}{n-1} f_i} 
        - \Delta f_i\bigg| 
       & =
        \int_{X_{i,\widehat\rho}} \chi_i \bigg( \frac{H_i}{1+ \frac{H_i}{n-1} f_i} - \Delta f_i \bigg)\\
&        \xrightarrow[i\to \infty]{} \int_{\overline{B}^X_{1+\widehat \rho} \setminus B^X_{1- \widehat \rho} }
        \chi \bigg( \frac{n-1}{\dist^X_o} - \frac{n-1}{\dist^X_o} \bigg) =0\, ,
    \end{split}
\end{equation}
where the convergence follows from \cite[Proposition 1.3.3(d)]{AmbrosioHondastab}. Hence \cite[Proposition 1.3.3(e)]{AmbrosioHondastab} implies that $\Delta f_i$ converges strongly in $L^2$.

Using Fermi coordinates as in the proof of \cref{thm:unique} and Fubini's theorem, it follows that the restriction of $\frac{H_i}{1+ \frac{H_i}{n-1} f_i} - \Delta f_i$ on almost every geodesic from $\partial \widetilde E_i$ tends to zero strongly in $L^1$ along the geodesic. Since $\Delta f_i$ are (uniformly) Lipschitz along geodesics from $\partial \widetilde E_i$ by the Riccati equation, then $\frac{H_i}{1+ \frac{H_i}{n-1} f_i} - \Delta f_i$ tends to zero strongly in $L^2$ (actually, uniformly) along almost every geodesic from $\partial \widetilde E_i$. In particular, by using the Dominated Convergence Theorem,
\begin{equation}\label{eq:zzFinale2}
    \begin{split}
        &\int_{\partial\widetilde E_{\widetilde r_i}} H_{\partial \widetilde E_{\widetilde r_i}}^2  
        \to
        (n-1)^2 n \theta\omega_n (1+ r_\infty)^{n-3}\, , 
        \qquad
        \int_{\partial \widetilde E_{\widetilde r_i}} H_{\partial \widetilde E_{\widetilde r_i}} \to
        (n-1) n \theta\omega_n (1+ r_\infty)^{n-2}\, .
    \end{split}
\end{equation}
Inserting \eqref{eq:zzFinale1}, \eqref{eq:zzFinale2} and the convergence properties for the profile $I_i$ collected above into \eqref{eq:zzFinale0} and taking the limit $i\to\infty$, we get a contradiction, thus completing the proof of \eqref{eq:zzConvergenzaUniformeFinale}.

\end{proof}

\appendix

\section{Quantitative regularity of almost minimizers in almost Euclidean manifolds}\label{appA}

The goal of this appendix is to discuss a series of effective regularity estimates for boundaries of isoperimetric sets in Riemannian manifolds satisfying suitable curvature bounds. These estimates play an important role in the context of the present paper. 

We note that similar effective regularity estimates for perimeter minimizing hypersurfaces and stable minimal or constant mean curvature hypersurfaces have been considered in the previous literature under different sets of assumptions. We address the reader for instance to \cite{RosenbergSouamToubiana,CheegerNabercurrents,ChodoshStrykerLi}, without the aim of being exhaustive in this list.
\medskip

Let us recall the following theorem, which comes directly from results of Anderson (see, e.g., \cite[Theorem 3.2]{Anderson}), and Cheeger--Colding \cite{CheegerColding96}.
\begin{thm}\label{thm:HarmonicRadius}
    For every $n\geq 2$, $\delta>0$, $C>0$, $\alpha\in [0,1)$ there exists $\varepsilon_0:=\varepsilon_0(n,\delta,C,\alpha)<1$ such that the following holds. 

    Let $(M^n,g,p)$ be a pointed $n$-dimensional Riemannian manifold such that 
    \begin{equation}
    -\varepsilon_0\leq \ric \leq C \quad \text{on $B_2(p)$}\, , 
    \qquad \mathrm{d}_{\mathrm{GH}}(B_2(p),B_2(0^n))\leq \varepsilon_0\, .
    \end{equation}
    Then there exists a (harmonic) chart around $p$ such that 
        \begin{equation}
        \|g_{ij}-\delta_{ij}\|_{C^{1,\alpha}(B_1(0^n))} \leq \delta\, .
        \end{equation}
\end{thm}

We shall also need an $\varepsilon$-regularity result for quasi-minimizers of the perimeter whose boundary has constant mean curvature in manifolds with controlled Ricci curvature. This will imply an $\varepsilon$-regularity result for isoperimetric sets in manifolds with nonnegative Ricci curvature and Euclidean volume growth, see \cref{rem:ApplicationToIsop}. Let us first recall the definition of $(\Lambda,r)$-minimizer. 

\begin{defn}\label{def:LambdaMin}
Let $(M,g)$ be a Riemannian manifold of dimension $n$, possibly non complete. Let $\Omega\subset M$ be open, and let $E\subset \Omega$ be a Borel set. Fix $\Lambda, r_0>0$. We say that $E$ is a \emph{$(\Lambda, r_0)$-minimizer in $(\Omega,g)$} if
\begin{equation}
P(E, B_r(x)) \le P(F, B_r(x))  + \Lambda \vol(E \Delta F)\, ,
\end{equation}
whenever $E \Delta F \Subset B_r(x)\Subset \Omega$ and $r \le r_0$.
\end{defn}

\begin{remark}\label{rem:LambdaMininimiCurvMedia}
Let $E$ be a $(\Lambda, r_0)$-minimizer in $(\Omega,g)$ for some open set $\Omega$ as in \cref{def:LambdaMin}. Then $P(E,\cdot)= \haus^{n-1} \res(\partial E \cap \Omega)$, and $\partial E$ has generalized mean curvature bounded by $\Lambda$. More precisely there exists a vector field $H \in L^1_{\rm loc}(\Omega,P(E,\cdot))$ such that
\begin{equation}
\int {\rm div}_\top X \de P(E, \cdot) = - \int g(X , H) \de P(E, \cdot)\, ,
\end{equation}
for any vector field $X \in C^\infty_c(\Omega)$, and $|H| \le \Lambda$ $P(E,\cdot)$-almost everywhere, see, e.g., \cite[Lemma A.2]{PascalePozzetta} for a proof in the Euclidean space (see also the proof of \cite[Theorema 4.7.4]{AmbrosioCorsoIntroduttivo}). Indeed, the first assertion follows by standard density estimates (see, e.g., the proof of \cite[Theorem 21.11]{MaggiBook}), while the existence of the generalized mean curvature vector follows by differentiation of the $\Lambda$-minimality condition.
\end{remark}

The following $\varepsilon$-regularity theorem has essentially appeared in \cite[Theorem 6.8]{MoS21}. The version we state here can be proved arguing verbatim.

\begin{prop}\label{prop:AdaptEpsRegularity}
Let $n \in \N$ with $n\ge 2$, $v_0>0$, $\Lambda>0$, and $\varepsilon>0$ be fixed. Then there exists $\delta>0$ for which the following holds.

Let $(M^n,g,p)$ be a pointed $n$-dimensional complete Riemannian manifold with $\mathrm{Ric}\geq -\delta$ on $B_2(p)$, and $|B_2(p)|\geq v_0$. Assume $E$ is a $(\Lambda,1)$-minimizer in $B_2(p)$, and
\begin{equation}
\max\left\{\dist_{\mathrm{GH}}\left(\partial E\cap B_2(p),B_2(0^{n-1})\right),\dist_{\mathrm{GH}}\left(B_2(p),B_2(0^n)\right),\left|\frac{|E\cap B_2(p))|}{\omega_{n}2^{n}}-\frac{1}{2}\right|\right\}<\delta \, .
\end{equation}
Then, for all $q\in B_{1/2}(p)\cap \partial E$, and $r<1$ it holds
\begin{equation}\label{eqn:Regularity}
\begin{split}
\dist_{\mathrm{GH}}\left(\partial E\cap B_r(q),B_r(0^{n-1})\right)&<\varepsilon r\, ,\\ 
\left|\frac{P(E,B_r(q))}{\omega_{n-1}r^{n-1}}-1\right|&<\varepsilon \, ,\\
\left|\frac{|E\cap B_r(q)|}{\omega_{n}r^{n}}-\frac{1}{2}\right|&<\varepsilon \, .
\end{split}
\end{equation}
\end{prop}

We are now ready to show an Allard-type result for $(\Lambda,1)$-minimizers in manifolds that are $C^{1}$-close to the Euclidean space, and to deduce an $\varepsilon$-regularity result for quasi-minimizer with constant mean curvature boundary in manifolds with controlled Ricci curvature, see \cref{thm:SecondeFormeConvergono}.

\begin{thm}[$C^{1,\gamma}$-regularity of almost flat $\Lambda$-minimizers on almost Euclidean balls]\label{C1gamma}
Fix $n \geq 2,\gamma\in(0,1), \Lambda>0$. Then there exist $\overline{\eps},\overline{\sigma} \in (0,1)$, and $\overline{C}>0$ such that the following holds.

Let $B_2(0^n)$ be endowed with a smooth Riemannian metric $g$ such that
\begin{equation}\label{eqn:C1Control}
\| g_{ij}- \delta_{ij}\|_{C^1(B_2(0^n))} \le \overline{\eps}\, .
\end{equation}
Let $E \subset B_2(0^n)$ be a $(\Lambda,1)$-minimizer in $(B_2(0^n), g)$. Suppose that $0^n \in \partial E$ and that
\begin{equation}\label{eqn:DensityControlPerimeter}
\left| \frac{P(E, B_r^{g}(0^n))}{\omega_{n-1}r^{n-1}} -1 \right| \le \overline{\eps}\, ,
\end{equation}
for any $r\in(0,1)$.
Then, up to rotation, there exists a Lipschitz function $u: (B_{\overline{\sigma}}(0^{n-1}), g_{\rm eu}) \to \R$ such that
\begin{equation}\label{eqn:HolderControls}
    |u|/\overline{\sigma} + |\nabla u | \le \overline{C}\, ,
\qquad
\nabla u(0)=0\, ,
\qquad
|\nabla u(x) - \nabla u(y)| \le  \overline{C} |x-y|^\gamma\, ,
\end{equation}
\begin{equation}\label{eqn:GraphConditions}
    \partial E \cap  B_{\overline{\sigma}}(0^n)  = \{ (x,u(x) ) \st x \in B_{\overline{\sigma}}(0^{n-1})\} \cap  B_{\overline{\sigma}}(0^n)\,  .
\end{equation}
\end{thm}

\begin{proof}
We first observe that there exist $\overline{\eps}_1\in(0,1)$ and $r_\Lambda=r_\Lambda(\Lambda, n) \in (0,1)$ such that if $g$ is a metric on $B_2(0^n)$ with $\| g_{ij}- \delta_{ij}\|_{C^1(B_2(0^n))} \le \overline{\eps}_1$, then any $(\Lambda,1)$-minimizer $G$ in $(B_2(0^n), g)$ satisfies density estimates of the form
\begin{equation}\label{eq:DensityEstimatesLambdaMin}
v_n \le \frac{\vol(G \cap B_\rho^g(x) )}{\omega_n \rho^n} \le 1- v_n\, ,
\qquad
c_n^{-1} \le \frac{P(G, B^g_\rho(x))}{\omega_{n-1}\rho^{n-1}} \le c_n\, ,
\end{equation}
for some $v_n \in (0,1)$ and $c_n> 1$ independent of $G, g$, for any $x \in \partial G \cap B_2(0^n)$ and any $\rho \le r_\Lambda$ with $B_\rho(x) \Subset B_2(0^n)$ (see, e.g., the proof of \cite[Theorem 21.11]{MaggiBook}).
\smallskip

The proof of the claim will eventually follow from an application of Allard's regularity theorem for varifolds having bounded first variation with respect to a parametric integrand, see \cite[The Regularity Theorem p. 27]{AllardProceedings86}. We sketch here the main steps leading to the application of the previous result.
\smallskip

It readily follows from the area formula that the perimeter measure of a set $G\subset B_2(0^n)$ with respect to a metric $g$ on $B_2(0^n)$ can be written as $P(G, \cdot) = \big( \scal{{\rm Cof}_x(n_G), n_G} \big)^{\frac12} P_{\rm eu}(G, \cdot)$, where $\scal{\cdot,\cdot}$ denotes Euclidean scalar product, $P_{\rm eu}(G, \cdot)= \haus^{n-1}_{\rm eu} \res \partial^*G$ is the Euclidean perimeter measure, $n_g$ is the unit Euclidean inner normal to $G$, and ${\rm Cof}_x$ is the cofactor matrix of the matrix $(g_{ij})$ at the point $x$.
Hence, if $\overline{\eps}_1\in(0,1)$ is small enough, the perimeter functional with respect to a metric $g$ with $\| g_{ij}- \delta_{ij}\|_{C^1(B_2(0^n))} \le \overline{\eps}_1$ is an admissible functional corresponding to the parametric integrand $\Psi(x,u) = \big( \scal{{\rm Cof}_x(u), u} \big)^{1/2}$, in the notation of \cite[3 p. 18]{AllardProceedings86}. Here by admissible we mean that all the constants $M_{i,j}, N_{i,j}$ appearing in \cite{AllardProceedings86} corresponding to $\Psi$ are bounded in terms of the dimension $n$ and of $\overline{\eps}_1$ only.

Take $A=0$ and $\eps=\lambda=\tfrac12$ in \cite[2.2]{AllardProceedings86} and recall \eqref{eq:DensityEstimatesLambdaMin}. Note that \cref{rem:LambdaMininimiCurvMedia} implies that the first variation with respect to the integrand $\Psi$ is bounded by $\Lambda$ in the sense of \cite[3.6(4)]{AllardProceedings86}. Taking into account \cite[Remark p. 28]{AllardProceedings86}, \cite[The Regularity Theorem p. 27]{AllardProceedings86} yields a parameter $\delta_8\in(0,1)$ (in the notation of \cite{AllardProceedings86}) and constants $\overline{C}=\overline{C}(n,\gamma), c(\Lambda)$ such that if $G \subset B_2(0^n)$ is a $(\Lambda,1)$-minimizer for a metric $g$ with $\| g_{ij}- \delta_{ij}\|_{C^1(B_2(0^n))} \le \overline{\eps}_1$ and
\begin{equation}\label{eq:IpotesiAllard}
    \frac12 <\frac{P(G,B_{\rho/2}(0^n))}{\omega_{n-1}(\rho/2)^{n-1}} < \frac32,
    \qquad
    |\mathbf{q}x| < \delta_8 \rho \quad \forall\, x \in \partial G \cap B_\rho(0^n)\, ,
\end{equation}
for some $\rho< c(\Lambda)\le 1$, then $\partial G \cap  B_{\overline{\sigma}_2\rho}(0^n)$ is parametrized by the graph of a $C^{1,\gamma}$ function $u$ with bounds as in the statement, for suitable $\overline{\sigma}_2=\overline{\sigma}_2(n, \gamma)\in(0,\tfrac12)$. Above $\mathbf{q}:\R^n\to\R$ is the projection $\mathbf{q}(x_1,\ldots, x_n):= x_n$. We are therefore left to verify \eqref{eq:IpotesiAllard} for a set $E$ satisfying, for suitable choices of $\overline{\eps}, \overline{\sigma}$.
\smallskip

Let $p>1$ be such that $\gamma=1-(n-1)/p$ and let $\delta_0, \overline{\sigma}_1 \in (0,1)$ be the two parameters, depending on $n, p$ only, in the assumptions in the classical Euclidean Allard regularity theorem \cite[Theorem 5.2 p. 121]{SimonBook}. Here $\overline{\sigma}_1$ denotes the constant called $\gamma$ in \cite[Theorem 5.2 p. 121]{SimonBook}. Define $\delta:= \tfrac12\min\{\delta_0, (\delta_8 \overline{\sigma}_1/(5C'))^{2n+2}, \Lambda c_n^{\frac1p}\omega_{n-1}^{\frac1p}\}$ where $C'=C'(n,\gamma)$ is the constant called $C$ in \cite[Theorem 5.2 p. 121]{SimonBook}, and $r_\delta:= \delta/( \Lambda c_n^{\frac1p}\omega_{n-1}^{\frac1p} ) < 1$. We finally define
\begin{equation*}
    \overline{\sigma}:= \frac{1}{6} \min\left\{\overline{\sigma}_1\overline{\sigma}_2 r_\delta, c(\Lambda) \overline{\sigma}_2 \right\}.
\end{equation*}
We want to show that there exist $\overline{\eps}\le \overline{\eps}_1$ such that if $g, E$ are as in the assumptions, then $E$ satisfies \eqref{eq:IpotesiAllard} with $\rho=\overline{\sigma}/\overline{\sigma}_2$, leading to the proof.\\
Suppose by contradiction that there exist $\overline{\eps}_i \le \overline{\eps}_1$ with $\overline{\eps}_i\searrow0$, and a sequence of metrics $g_i$ and a sequence of $(\Lambda,1)$-minimizers $E_i$ as in the assumptions, but \eqref{eq:IpotesiAllard} with $G=E_i$ and with $\rho=\overline{\sigma}/\overline{\sigma}_2$ does not hold. By \eqref{eq:DensityEstimatesLambdaMin}, up to subsequence $E_i$ converges to a nonempty Euclidean $(\Lambda, 1)$-minimizer $F$. The fact that $F$ is a Euclidean $(\Lambda, 1)$-minimizer follows by easily adapting the proof of \cite[Theorem 21.14]{MaggiBook}. Moreover, by convergence of the perimeter measures as in \cite[Theorem 21.14]{MaggiBook} and applying \cref{rem:LambdaMininimiCurvMedia} on the set $F$, we get
\begin{equation}\label{eq:IpotesiAllClassicoSuLimite}
    \frac{P_{\rm eu}(F, B_r(0^n))}{\omega_{n-1}r^{n-1}} = 1\, ,
\qquad
\left( \int_{B_r(0^n)} |H_{\partial F}|^p \de P_{\rm eu}(F, \cdot) \right)^{\frac1p} r^{1- \frac{n-1}{p}} \overset{\eqref{eq:DensityEstimatesLambdaMin}}{\le} \Lambda c_n^{\frac1p}\omega_{n-1}^{\frac1p} r \le \delta\, ,
\end{equation}
for any $r\le r_\delta$. Hence the classical Euclidean Allard regularity theorem \cite[Theorem 5.2 p. 121]{SimonBook} can be applied to $F$, implying that, up to rotation, $\partial F \cap B_{\overline{\sigma}_1 r}$ is parametrized as the graph of a $C^{1,\gamma}$ function $f:B_{\overline{\sigma}_1 r}(0^{n-1})\to \R$ with $|f|/r + |\nabla f| \le C' \delta^{1/(2n+2)}$ and $[\nabla f]_{0,\gamma} \le  C' \delta^{1/(2n+2)} r^\gamma$, for any $r\le r_\delta$. Since $3 \overline{\sigma}/(\overline{\sigma}_1 \overline{\sigma}_2) < r_\delta$, we can choose $r=3 \overline{\sigma}/(\overline{\sigma}_1 \overline{\sigma}_2)$, so that there holds
\begin{equation}\label{eq:HeightSulLimite}
    |\mathbf{q}x| \le C' \delta^{1/(2n+2)} r < \frac35 \delta_8 \frac{\overline{\sigma}}{\overline{\sigma}_2}  \qquad \forall\, x \in \partial F \cap B_{3 \overline{\sigma}/\overline{\sigma}_2}(0^n)\, .
\end{equation}
Since $\partial E_i\to \partial F$ in Kuratowski sense thanks to the uniform density estimates (see, e.g., the proof of \cite[Theorem 21.14 (i), (ii)]{MaggiBook}), the first equalities in \eqref{eq:IpotesiAllClassicoSuLimite} and \eqref{eq:HeightSulLimite} imply that \eqref{eq:IpotesiAllard} hold on $E_i$ for $i$ large enough with $\rho= \overline{\sigma}/\overline{\sigma}_2$, thus yielding a contradiction.
\end{proof}

\begin{thm}\label{thm:SecondeFormeConvergono}
Let $n \in \N$ with $n\ge 2$, $C,\Lambda,v_0>0$, $H\in\mathbb R$, and $\alpha\in (0,1)$. Then there exist $\eps<1, \sigma \in (0,1)$, $B>0$ such that the following holds.

Let $(M^n,g,p)$ be a pointed $n$-dimensional complete Riemannian manifold with $-\varepsilon \leq \ric\leq C$ on $B_2(p)$, and $|B_2(p)|\geq v_0$. Suppose $E\subset M$ is a set such that
\[
\text{$E$ is a $(\Lambda,1)$-minimizer in $B_2(p)$, $\partial E$ has a.\!\! e.\!\! generalized mean curvature $H_{\partial E}\equiv H$}\, ,
\]
and
\begin{equation}\label{eqn:Control1AprioriSff}
\max\left\{\d_{\rm GH}(B_2(p), B_2(0^n)),\dist_{\rm GH}\left( \partial E \cap B_2(p) , B_2(0^{n-1}) \right), \left| \frac{|E \cap B_2(p)|}{\omega_n 2^n} - \frac12\right| \right\}
 \le \eps\, .
\end{equation}
Then, denoting with $\partial_r E$ the regular part of $\partial E$, it holds $\partial E \cap B_\sigma(p) \subset \partial_r E$, and
\begin{equation}\label{eqn:sffBound}
\left\| |\sff| \right\|_{C^{0,\alpha}(\partial E \cap B_\sigma(p), \dist|_{\partial E})} 
\eqdef \sup_{x \in \partial E \cap B_\sigma(p)} |\sff|(x) + \sup_{x, y \in \partial E \cap B_\sigma(p)} \frac{\left| |\sff|(x) -  |\sff|(y) \right|}{\dist(x,y)^\alpha}
\le B\, ,
\end{equation}
where $\sff$ denotes the second fundamental form of $\partial E$. In particular, $\partial E \cap B_\sigma(p)$ is smooth.

Moreover, if also $|{\rm Sect}| \le K$ on $B_2(p)$, then the (intrinsic) sectional curvature of $\partial E \cap B_\sigma(p)$ is bounded from below by a constant $B'=B'(B,K)\in\R$.
\end{thm}

\begin{proof}
   Let $\bar\varepsilon,\bar\sigma,\bar C>0$ be the constants depending on $\Lambda,\gamma:=\alpha,n$ for which \cref{C1gamma} holds. Let $\varepsilon_0:=\varepsilon_0(n,\bar\varepsilon,C,\alpha)$ be the constant for which \cref{thm:HarmonicRadius} holds.
    Thus, if $\varepsilon<\varepsilon_0$ in the statement of \cref{thm:SecondeFormeConvergono}, there exists a chart centered at $p$ such that
    \begin{equation}
    \|g_{ij}-\delta_{ij}\|_{C^{1,\alpha}(B_1(0^n))}\leq \bar\varepsilon\, .
    \end{equation}
    Finally let $\delta:=\delta(\bar\varepsilon,n,v_0,\Lambda)$ be the constant appearing in \cref{prop:AdaptEpsRegularity}. We claim that in the statement of \cref{thm:SecondeFormeConvergono} we can take $\varepsilon:=\min\{\delta,\varepsilon_0\}$.

    Indeed, by how we chose the constants, and subsequently applying \cref{thm:HarmonicRadius}, 
    \cref{prop:AdaptEpsRegularity}, and \cref{C1gamma} we get from \eqref{eqn:GraphConditions} that $\partial E\cap B_{\bar\sigma/2}(p)$ is a graph of a $C^{1,\alpha}$-function $u$ defined on some ball $B_{r}(0^{n-1})\subset B_{r}(0^n)$, with $\bar\sigma/4<r<\bar\sigma$, contained in the domain of the chart given by \cref{thm:HarmonicRadius}. Thus $\partial E\cap B_{\sigma}(p)\subset \partial_r E$ with $\sigma:=\bar\sigma/2$, where $\partial_rE$ denotes the regular part of $\partial E$.
\medskip
    
    In order to obtain the H\"older bound on the second fundamental form we exploit a standard PDE argument based on the fact that $E$ has smooth boundary and it has constant mean curvature. Let $u:B_r(0^{n-1})\to \mathbb R$, with $\bar\sigma/4 < r < \bar\sigma$, be the function defined above such that the graph $(x,u(x))$ in the chart given by \cref{thm:HarmonicRadius} is (a subset of) $\partial E$. Let us define
    \begin{equation}
    h_{ij}:=g_{ij}+\partial_i u \cdot g_{jn}+\partial_j u \cdot g_{in} + \partial_i u\cdot \partial_j u \cdot g_{nn}\, , \qquad \forall i,j=1,\ldots,n-1\, .
    \end{equation}
    Moreover, let $\{e_1,\ldots,e_{n}\}$ be the canonical basis in $\mathbb R^{n}$, and let $v_i:=(e_i,\partial_i u)\in \mathbb R^n$, for every $i=1,\ldots,n-1$. Denote $\{\hat v_i\}_{i=1,\ldots,n-1}$ the orthonormal vectors obtained by applying Gram-Schmidt reduction to $\{v_i\}_{i=1,\ldots,n-1}$. Denote $e_n^\perp:=e_n-\sum_{i=1}^{n-1} g(e_n,\hat v_i)\hat v_i$.
    
    A routine computation shows that the mean curvature of the graph of $u$ in the chart is 
    \begin{equation}\label{eqn:Elliptic}
    H=g(e_n^\perp,e_n^\perp)^{1/2}\cdot h^{ij}\cdot \partial_{ij}u + Q(g,\partial g, \partial u)\, ,
    \end{equation}
    where $Q(g,\partial g,\partial u)$ is a rational expression involving $g,\partial g,\partial u$.

    Hence, if $\|g_{ij}-\delta_{ij}\|_{C^0}+\|u\|_{C^1}\leq \varepsilon(n)$, where $\varepsilon(n)$ is a universal small constant, \eqref{eqn:Elliptic} is an elliptic equation. By taking $\bar\varepsilon\leq \varepsilon(n)/2$, which we can always do without loss of generality, and by taking into account \eqref{eqn:C1Control}, and the second inequality in \eqref{eqn:HolderControls}, we have the following. There exists $N:=N(n,\alpha,\Lambda)>1$ large enough such that 
    \begin{equation}
    \|g_{ij}-\delta_{ij}\|_{C^0(B_{r/N}(0^n))}+\|u\|_{C^1(B_{r/N}(0^{n-1}))} \leq \varepsilon(n)\, .
    \end{equation}
    
    Hence \eqref{eqn:Elliptic} is an elliptic equation on $B_{r/N}(0^{n-1})$. We can thus apply Schauder estimates to get
    \begin{equation}\label{eqn:C2alphabound}
    \|u\|_{C^{2,\alpha}(B_{r/(2N)}(0^{n-1}))}\leq C\left(r,N,\alpha,|H|,\|g\|_{C^{1,\alpha}(B_{r/N}(0^{n}))},\|u\|_{C^{1,\alpha}(B_{r/N}(0^{n-1}))}\right)\, .
    \end{equation}
    Finally, writing $\sff$ in the chart $\varphi$, we get \eqref{eqn:sffBound} up to shrinking $\sigma:=\bar\sigma/(16N)$, and where $B$ only depends on $\bar\varepsilon$, $r$, $N$, $\alpha$, $\bar C$, and an upper bound on $|H|$, and thus ultimately only on $\alpha,n,\Lambda,H$.
\medskip

    The last assertion about the sectional curvature with respect to the induced metric being uniformly bounded readily comes from the Gauss equations, taking into account the uniform bounds for the second fundamental form and the bounds for the sectional curvature of the ambient Riemannian manifold.
\end{proof}

\begin{remark}[Applying \cref{thm:SecondeFormeConvergono} to isoperimetric sets]\label{rem:ApplicationToIsop}
    \cref{thm:SecondeFormeConvergono} has, as a direct consequence, an $\varepsilon$-regularity result for isoperimetric sets with large volumes in complete manifolds with $\ric\ge 0$, $\mathrm{AVR}>0$, and quadratic curvature decay \eqref{eq:QCD}, which is relevant to our aims.

    Let $(M^n,g)$ be a smooth Riemannian manifold with $\mathrm{Ric}\geq 0$, and with $\mathrm{AVR}\geq v_0>0$. 
Choosing $V_1=\frac 12$, $V_2=2$, $V_3=3$, $K=0$ in \cite[Corollary 4.15]{AntonelliPasqualettoPozzettaSemola1}, and by a scaling argument, we get that there is $\mathcal{C}:=\mathcal{C}(v_0,n)\geq 1$, such that 
\begin{equation}\label{eqn:EverySetLambdaMin}
\text{Any isoperimetric set $E\subset M$ with $|E|\geq 1$ is a $(\mathcal{C}|E|^{-1/n},\mathcal{C}^{-1}|E|^{1/n})$-minimizer}\, ,
\end{equation}
and moreover
\begin{equation}
\mathrm{Per}(E,B_r(p))\leq (1+\mathcal{C}|E|^{-1/n}r)\mathrm{Per}(F,B_r(p))\, ,
\end{equation}
for every $p\in M$, $r<\mathcal{C}^{-1}|E|^{1/n}$, and $F$ such that $E\Delta F\subset\subset B_r(p)$.
    
Now, let $V_1\geq \mathcal{C}^n$ be fixed, and define  $\Lambda:=\mathcal{C}V_1^{-1/n}$. By \eqref{eqn:EverySetLambdaMin}, every isoperimetric set $E$ with $|E|\geq V_1$ in $M$ is a $(\mathcal{C}|E|^{-1/n},\mathcal{C}^{-1}|E|^{1/n})$-minimizer, and thus in particular a $(\Lambda,1)$-minimizer. Moreover, recall that the boundary of a smooth isoperimetric set has constant mean curvature, and if $|E|\geq V_1$ in $M$, then $|H_{\partial E}|\leq C(n,V_1,v_0)$, see \cite[Corollary 3.5]{AntonelliPasqualettoPozzettaSemola2}.
    
Taking into account also \cref{lemma:hausconv}, and the regularity of cross-sections of manifolds with $\ric\ge 0$, $\mathrm{AVR}>0$ and quadratic curvature decay \eqref{eq:QCD}, we infer that\cref{thm:SecondeFormeConvergono} can be applied in this setting, as soon as $E$ is an isoperimetric set with $|E|\geq V_1\geq \mathcal{C}(v_0,n)^n$. In this case the constants $\varepsilon,\sigma,B$ will obviously depend on $n,C,V_1,v_0$. 
\end{remark}

\section{Uniform bounds for the reach}

The goal of this appendix is to prove that the reach of a smooth hypersurface $\Sigma^{n-1}$ embedded in a smooth Riemannian manifold $M^n$ can be bounded from below in terms of the second fundamental form of the hypersurface and the norm of the sectional curvature of the ambient manifold, under some minor additional technical assumptions. The statement is probably known to experts, and it follows for instance from a slight adaptation of the arguments in \cite{HeintzeKarcher}. We provide a detailed proof as we are not able to locate a precise reference in the literature. The argument is due to Petrunin, see \cite{PetruninMO}.

\begin{prop}\label{prop:reachbound}
Let $E$ be a compact subset of an $n$-dimensional smooth Riemannian manifold $M^n$, and assume that $\partial E:=\Sigma^{n-1}$ is a codimension one 
hypersurface embedded in $M^n$. Let $\nu_\Sigma$ be the unit inner normal. Assume 
there is $\varepsilon_0>0$ such that 
\begin{equation}\label{eqn:2SidedWell}
    \exp_p(t\nu_\Sigma(p)) \in 
    \begin{cases}
        E \quad &\text{for all $t\in (0,\varepsilon_0)$}\, ,\\
        M\setminus E \quad &\text{for all $t\in (-\varepsilon_0,0)$}\, .
    \end{cases}
\end{equation}
Let $C>0$ be such that 
\begin{equation}\label{eqn:HypothesesOnSffSec}
\sup_{x\in \Sigma}|\rm{II}(x)|\le C\, ,\quad \sup_{x\in B_1(\Sigma)}|\mathrm{Riem}(x)|\le C\, ,
\end{equation}
where we denote $B_r(\Sigma):=\{x\in M\,: \d(x,\Sigma)<r\}$ for each $r>0$. Then there exists a constant $r_0=r_0(C,n,\varepsilon_0)>0$ such that
\begin{equation}\label{eqn:SoughtDiffeomorphism}
    F:(-r_0,r_0)\times \Sigma \to M\, , \qquad F(t,p):=\exp_p(t\nu_\Sigma(p))
\end{equation}
is a diffeomorphism with its image.
\end{prop}

\begin{proof}
Let us take $\mathcal{C}$ a constant chosen small enough with respect to $C$, and $n$. We claim that the choice $r_0:=\min\{\frac{\varepsilon_0}{4},\mathcal{C}\}$ works. First, $\mathcal{C}$ can be chosen small enough so that, since $r_0<\mathcal{C}$, by using \eqref{eqn:HypothesesOnSffSec} and classical estimates on Jacobi fields, one gets that 
\[
\mathrm{d}F_{(t,p)}(v) \neq 0 \quad \text{for every $(t,p)\in (-r_0,r_0)\times\Sigma$, and every $0\neq v \in T_{(t,p)}\left((-r_0,r_0)\times\Sigma\right)$}\, .
\]

Thus for every $(t,p)\in (-r_0,r_0)\times\Sigma$ we have that $\mathrm{d}F_{(t,p)}$ is a linear isomorphism, and then $F$ is a local diffeomorphism around every point $(t,p)\in (-r_0,r_0)\times\Sigma$. In order to prove \eqref{eqn:SoughtDiffeomorphism} it is left to prove that $F$ is injective. 

Suppose by contradiction $F$ is not injective, and consider 
\begin{equation}\label{eqn:MinimalityOfL}
\ell:=\inf\{0<\vartheta\leq r_0: F:(-\vartheta,\vartheta)\times\Sigma\to M \quad \text{is not injective}\}>0\, .
\end{equation}
Hence there are $p,q\in\Sigma$ such that setting $\gamma_p(t):=\exp_p(t\nu_\Sigma(p))$, and $\gamma_q(t):=\exp_q(t\nu_\Sigma(q))$ we have
\[
x=\gamma_p(\ell)=\gamma_q(\ell)\, .
\]
We claim that $\gamma_p'(\ell)=-\gamma'_q(\ell)$. Indeed, if not, by taking a geodesic starting from $x$ and forming acute angles with $\gamma_p,\gamma_q$ at $\ell$, one would get that there is a point $z$ such that $d(z,p)<\ell$, and $d(z,q)<\ell$. Thus, by possibly taking $z$ closer to $x$, and by exploiting the minimality of $\ell$, $z$ can be chosen so that $z=\exp_{p'}(\ell_p\nu_\Sigma(p'))=\exp_{q'}(\ell_q\nu_\Sigma(q'))$ for $p',q'\in \Sigma$, with $p'\neq q'$, and $\ell_p,\ell_q<\ell$. This last property would contradict the minimality of $\ell$ in \eqref{eqn:MinimalityOfL}.

Hence, since $\gamma'_p(\ell)=-\gamma'_q(\ell)$, we get that $q=\exp_p(2\ell\nu_\Sigma(p))$. Thus, by the second property in \eqref{eqn:2SidedWell} applied at $q$, we have that 
\[
\exp_p((2\ell+\varepsilon_0/3)\nu_\Sigma(p))\in M\setminus E\, ,
\]
which is a contradiction with the first of \eqref{eqn:2SidedWell} applied at $p$ since $2\ell+\varepsilon_0/3 \leq 2r_0+\varepsilon_0/3<\varepsilon_0$.   
\end{proof}

\bigskip
\noindent\textbf{Data availability statement.} The present paper has no associated data.

\medskip
\noindent\textbf{Conflict of interest.} The authors declare no conflict of interest.

\medskip
\noindent\textbf{Consent.} The authors agreed with the content of the paper and they gave consent to the present submission.

\printbibliography
\end{document}